 \documentclass[final]{IEEEtran}

\usepackage{url}
\usepackage{cite}
\usepackage{epsf}
\usepackage{times}
\usepackage{graphicx}
\usepackage[english]{babel}
\usepackage{color}
\usepackage{amssymb}  %
\usepackage{subfigure}  %
\usepackage{bm}
\usepackage{amsmath}
\usepackage{makeidx}  
\usepackage{times}
\usepackage{algorithm,algorithmic,xspace}
\usepackage{amsthm}  %
\usepackage{soul}  

\newcommand{\E}{\mathrm{E}}
\newcommand{\VAR}{\mathrm{VAR}}
\newcommand{\RMSE}{\mathrm{RMSE}}

\DeclareMathOperator*{\argmin}{arg\,min}

\newcommand{\eqdef}{ \stackrel{\textrm{\tiny def}}{=}}

\renewcommand{\natural}{{\mathbb{N}}}

\newcommand{\integernonnegative}{{\mathbb{Z}_{\geq0}}}
\newcommand{\naturalzero}{\mathbb{N}_0}
\renewcommand{\naturalzero}{\integernonnegative}

\newcommand{\real}{{\mathbb{R}}}

\newcommand{\union}{\cup}

\newcommand{\until}[1]{\{1,\dots,#1\}}

\newcommand{\nmax}{n_{\text{max}}}
\newcommand{\bML}{\hat{b}^{\text{ML}}}
\newcommand{\bhom}{\hat{b}^{\text{hom}}}
\newcommand{\lambdaEB}{\hat{\lambda}^{\text{EB}}}
\newcommand{\lambdaAH}{\hat{\lambda}^{\text{ad-hoc}}}
\newcommand{\lambdaDEC}{\hat{\lambda}^{\text{dec}}}

\newcommand{\EBestimator}{Empirical Bayes distributed estimator}
\newcommand{\AHestimator}{ad-hoc distributed estimator}

\newcommand{\sigmavec}{\bm{\sigma}}
\newcommand{\nvec}{\bm{n}}
\newcommand{\lambdavec}{\bm{\lambda}}

\newtheorem{theorem}{Theorem}[section]
\newtheorem{proposition}[theorem]{Proposition}

\newtheorem{lemma}[theorem]{Lemma}
\newtheorem{remark}[theorem]{Remark}

\newtheorem{assumption}[theorem]{Assumption}

\newcommand\oprocendsymbol{\hbox{$\square$}}
\newcommand\oprocend{\relax\ifmmode\else\unskip\hfill\fi\oprocendsymbol}

\renewcommand{\theenumi}{(\roman{enumi})}

\graphicspath{{figs/}}

\begin{document}

\title{A Bayesian framework for distributed estimation\\ of arrival rates in asynchronous networks}

\author{Angelo Coluccia, {\it member, IEEE} and Giuseppe Notarstefano, {\it member, IEEE} \thanks{Copyright (c) 2015
    IEEE. Personal use of this material is permitted. However, permission to use
    this material for any other purposes must be obtained from the IEEE by
    sending a request to pubs-permissions@ieee.org. Angelo Coluccia and
    Giuseppe Notarstefano are with the Department of Engineering, Universit\`a
    del Salento, via Monteroni, 73100, Lecce, Italy,
    \texttt{\{name.lastname\}@unisalento.it.}  } \thanks{A preliminary short
    version of this paper has appeared as \cite{coluccia2014distributed}. This
    result is part of a project that has received funding from the European
    Research Council (ERC) under the European Union’s Horizon 2020 research and
    innovation programme (grant agreement No 638992 - OPT4SMART).}}

\maketitle 


\begin{abstract}
  In this paper we consider a network of agents monitoring a spatially
  distributed arrival process. Each node measures the number of arrivals seen at
  its monitoring point in a given time-interval with the objective of estimating
  the unknown local arrival rate. 
  We propose an asynchronous distributed approach based on a Bayesian model with
  unknown hyperparameter, where each node computes the minimum mean square error
  (MMSE) estimator of its local arrival rate in a distributed way. As a result,
  the estimation at each node ``optimally'' fuses the information from the whole
  network through a distributed optimization algorithm. Moreover, we propose an
  \emph{ad-hoc} distributed estimator, based on a consensus algorithm for
  time-varying and directed graphs, which exhibits reduced complexity and
  exponential convergence.
  We analyze the performance of the proposed distributed estimators, showing
  that they: (i) are reliable even in presence of limited local data, and (ii)
  improve the estimation accuracy compared to the purely decentralized setup.
  Finally, we provide a statistical characterization of the proposed estimators.
  In particular, for the ad-hoc estimator, we show that as the number of nodes
  goes to infinity its mean square error converges to the optimal one.
  Numerical Monte Carlo simulations confirm the theoretical characterization and
  highlight the appealing performances of the estimators.
\end{abstract}

\begin{keywords}
  distributed estimation, Empirical Bayes, push-sum
  consensus, cyber-physical systems. 
\end{keywords}

\section{Introduction}
\label{introduction}

Arrival processes provide a useful description for events occurring with some
probability in a given time or space interval. Applications range from
communications and transports to medicine (e.g., diagnostic imaging) and
astronomy (e.g., particle detection), \cite{snyder1991random}.
From a statistical point of view the prominent model for arrival processes is
notoriously Poisson. Indeed, even when independence between arrivals cannot be
assumed, the superposition of a large number of non-Poisson processes is
approximately distributed as a Poisson (Palm-Khintchine Theorem),
\cite{heyman2003stochastic}. Also, the limit distribution of counting processes
described by the Binomial distribution is a Poisson, according to the law of
rare events.

In modern network contexts (as, e.g., data, communication and sensor networks or
Intelligent Transportation Systems) estimating the process intensity at
different locations, i.e., the arrival rates at different nodes, is an important
preliminary problem to be solved in order to gain context awareness. 
Clearly, each arrival rate can be estimated by performing a
\emph{decentralized} Maximum Likelihood (ML) estimation based only on the local
arrivals at the given node. In this paper we want to investigate how the
estimation can be improved by exploiting the cooperation among the nodes.
Recently, a great interest has been devoted to cooperation schemes in which
estimation is performed by a network of computing nodes in a distributed way,
rather than by collecting all the data in a central unit,
\cite{barbarossa2013distributed,schizas2008consensus,cattivelli2010diffusion}.
How to use the information from other nodes in the network and how to design a
distributed algorithm merging such information will be the focus of the paper.

Distributed estimation has received a widespread attention in the distributed
computation literature, especially as a natural application of linear (average)
consensus algorithms, see, e.g. \cite{barbarossa2007bio}, or the recent surveys
\cite{garin2011survey,frasca2015distributed}.
Nodes typically 
interact iteratively with their neighbors by means of a ``diffusion-like''
process in which the estimation is improved by suitably combining the
estimations from neighboring nodes,
\cite{cattivelli2010diffusion,sardellitti2010fast,marelli2015distributed}.
An incremental and a diffusive distributed algorithm with finite time
convergence are proposed in \cite{pasqualetti2012distributed} for (static) state
estimation.
Distributed optimization is also strictly related to distributed estimation. In
\cite{schizas2008consensus} a distributed Alternating Direction Method of
Multipliers (ADMM) has been introduced as a tool for distributed ML estimation
of vector parameters in a wireless sensor network.
Notice that ADMM and other distributed optimization algorithms, as, e.g.,
\cite{nedic2013distributed}, can be used as building blocks to solve parts of an
estimation problem in a distributed set-up.
ML approaches for distributed estimation of a commonly observed parameter have
been proposed also in \cite{barbarossa2007decentralized}. %
In \cite{chiuso2011gossip,fagnani2011input} consensus-based algorithms have been
developed to simultaneously estimate a common parameter measured by noisy
sensors and classify sensor types.
In \cite{varagnolo2010distributed} consensus-based algorithms have been
developed to estimate global parameters in a linear Bayesian framework.
In \cite{dilorenzo2014distributed} a distributed algorithm is proposed for adaptive
Bayesian estimation of a common parameter with known prior. 
Dynamic methods have been also proposed in which the nodes keep collecting new
measurements while interacting with each other.
In \cite{cattivelli2008diffusion} a diffusion-based Recursive Least Squares (RLS)
algorithm is proposed to estimate a constant parameter, but with dynamically
acquired measurements. 
Finally, in \cite{mateos2009distributed} and \cite{schizas2008consensus2}
distributed ADMM-based algorithms are proposed for the estimation of random
signals and dynamical processes.   

Differently from the above references, in our work we consider a more general
Bayesian framework that allows the nodes to improve their local estimate, rather
than reaching a consensus on a common parameter. 
As detailed later, we will consider a special model for continuous mixtures of
Poisson variables. Poisson mixtures have been widely used in the arrival-rate
estimation literature to model non-homogeneous scenarios, see, e.g.,
\cite{freedman1962poisson,massey1996estimating} as early references. The survey
\cite{karlis2005mixed} provides an extensive review of properties and
applications for Poisson mixtures. In a Bayesian setting the classical mixing
(prior) distribution is the Gamma \cite{karlis2005mixed,withers2011compound},
which results in a closed-form posterior \cite{casella}.

The main contribution of the paper is twofold. First, we propose a distributed
estimation scheme for arrival rates in an asynchronous network, based on a
hierarchical probabilistic framework.
Specifically, we develop an Empirical Bayes approach, in which the arrival rates
are treated as random variables, whose prior distribution is parametrized by an
unknown hyperparameter to be determined via ML estimation.
In particular, we borrow from the centralized statistics literature the
  classical Gamma-Poisson model (for Poisson mixtures), assuming however the
  hyperparameter is unknown, and extend it to a scenario with non-homogeneous
  sample sizes. We show that the ML estimation of the hyperparameter is a
  separable optimization problem, that can be solved in a distributed way over
  the network by using a distributed optimization algorithm.
Thanks to this modeling idea, the local estimates are
obtained by taking advantage of the whole network data, thus improving the
accuracy especially when the amount of local data is scarce.  
With this approach we are able to capture the fact that arrival rates are not
the outcomes of isolated phenomena, but rather the expression of global
properties of the process. 
For this ``optimal'' estimator we characterize mean and variance at steady-state
(i.e., once  agents have reached consensus on the optimal solution) for
networks with a large number of agents.

Second, we propose an alternative ad-hoc distributed estimator that, although
suboptimal, performs comparably to the optimal one. The main advantage of this
ad-hoc estimator is that the resulting algorithm has a simple update rule based
on linear consensus protocols, thus exhibiting the same appealing exponential
convergence. 
For the ad-hoc estimator we also characterize transient (at any time-instant)
mean and variance for a given number of agents.
Notably, we show that at steady-state and for large number of agents the ad-hoc
estimator attains the performance of the optimal one.

To strengthen these two contributions we perform a Monte Carlo analysis and
compare the theoretical expressions obtained for mean and variance with their
sample counterparts. The numerical computations confirm the theoretical analysis
and show that some key assumptions made for a rigorous, but tractable analysis
are not restrictive. Moreover, they highlight other interesting features of the
proposed distributed estimators. For example, the ad-hoc estimator achieves
performances close to the optimal one even for a limited number of nodes.

The paper is organized as follows. In Section~\ref{sec:prob_form} we introduce
the model of a monitoring network and set up the estimation problem. In
Section~\ref{sec:estimators} we develop the proposed distributed estimators
based on the Empirical Bayes approach. The statistical performances of the two
estimators are analyzed in Section~\ref{sec:theoretical_analysis}, while in
Section~\ref{sec:numerical_analysis} we perform a Monte Carlo analysis to
confirm the theoretical results.

\section{Monitoring network and arrival rate estimation problem}
\label{sec:prob_form}
We consider a \emph{network of monitors} having \emph{sensing},
\emph{communication} and \emph{computation} capabilities. That is, each monitor
can measure the number of arrivals in a given measurement time-scale
(e.g., 1 second or 1 minute), share local data with neighboring agents,
and perform local computations on its own and its neighbors' data. The objective
is to fuse the data in order to improve the estimation of the arrival rates.

Formally, each node collects measurements \emph{asynchronously} in an
\emph{observation window}, over which the underlying process can be assumed to
be stationary, i.e., the set of rates can be considered constant over the
  observation window. Clearly, the arrival rates have stationary increments, so
  that the number of arrivals in disjoint intervals are statistically
  independent.
A measurement consists of the number of arrivals detected at a given location in
the (common) \emph{time-scale interval}.
Accordingly, for each monitor $i\in\{1, \ldots, N\}$ we introduce the following
variables:
\begin{itemize}
\item $\lambda_i$ unknown arrival rate;
\item $y_{i,\ell}$ the $\ell$-th collected measurement (number of arrivals per
  time-scale interval);
\item $n_i \in [1,\nmax]$ number of measurements $y_{i,\ell}$ collected in the
  observation window (where $\nmax$ is the maximum number of measurements that
  can be collected).
\end{itemize}
The conditional distribution of $y_{i,\ell}$  given $\lambda_i$
is a Poisson random variable with parameter $\lambda_i$, i.e.,
$y_{i,\ell}| \lambda_i \sim \text{Poisson} (\lambda_i)$. All measurements are
assumed to be independent. 

We denote by $n \eqdef \sum_{i=1}^N n_i$ the total number of measurements. If all
the nodes have the same number of measurements, i.e., $n_i = n/N$ for all
$i\in\until{N}$ we say that the network is \emph{homogeneous}.

In Figure~\ref{fig:measurement_scenario} a scheme of the network measurement
scenario is depicted with the variables of interest.

 \begin{figure}[h]
   \centering
   \includegraphics[width=0.35\textwidth]{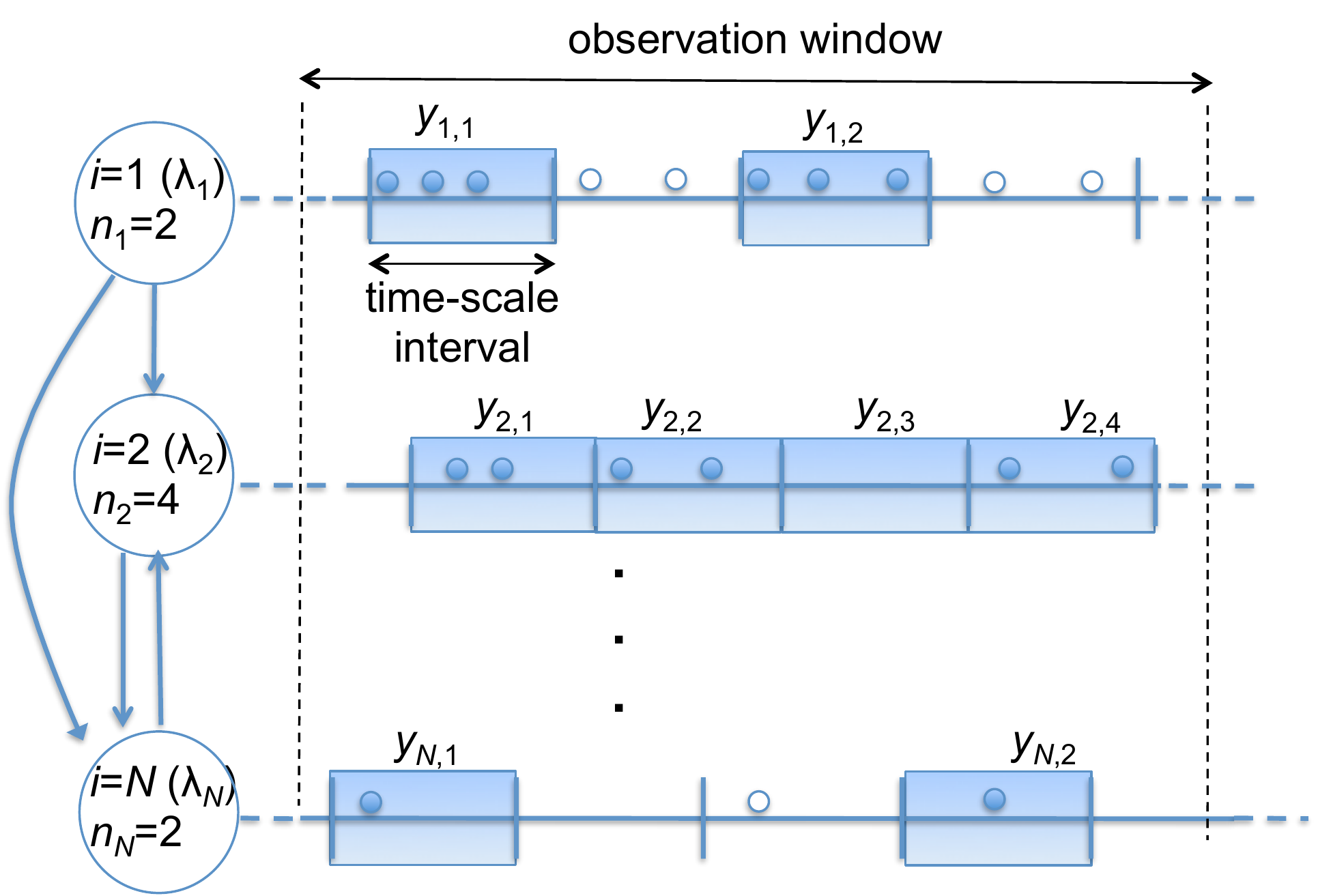}
   \caption{Scheme of the measurement scenario.}
   \label{fig:measurement_scenario}
 \end{figure}

We assume that the network evolution is triggered by a \emph{universal slotted
  time}, $t\in \naturalzero$, not necessarily known by the monitors.
The monitors communicate according to a time-dependent directed communication
graph $t \mapsto G(t) = (\until{N}, E(t))$, where $\until{N}$ are the monitor
\emph{identifiers} and the edge set $E(t)$ describes the communication among
monitors: $(i,k)\in E(t)$ if monitor $i$ communicates to $k$ at time
$t\in\naturalzero$. 
For each node $i$, the nodes sending information to $i$ at time $t$, i.e., the
set of $k\in \until{N}$ such that $(k,i)\in E(t)$ is the set of \emph{in-neighbors} of $i$
at time $t$, and is denoted by $N^{I}_i(t)$. 
We make the following minimal assumption on the graph connectivity. First,
  we recall that a fixed directed graph is strongly connected if for any pair of
  nodes, $i$ and $j$, there exists a directed path (i.e., a set of consecutive
  edges) from $i$ to $j$.
\begin{assumption}[Uniform joint strong connectivity]~
\label{assum:graph}
  There exists an integer $Q\ge1$ such that the graph
  $\left(\{1, \ldots, N\}, \bigcup_{\tau=tQ}^{(t+1)Q-1} E(\tau)\right)$ is strongly
  connected $\forall \, t\ge0$.
\end{assumption}
It is worth remarking that this network setup is very general, since it
naturally embeds asynchronous scenarios as well as missing measurements due to,
e.g., sensor failures.

\section{Distributed arrival rate estimation\\ via Empirical Bayes}
\label{sec:estimators}
In a \emph{decentralized set-up}, in which nodes do not communicate, each node
could estimate $\lambda_i$ based on the sample
$\{y_{i,\ell}\}_{\ell=1,\ldots,n_i}$ by simply computing the empirical mean of
the available measurements. That is, the decentralized estimator is
$$
\lambdaDEC_i
= \frac{1}{n_i} \sum_{\ell = 1}^{n_i} y_{i,\ell} = \frac{\sigma_i}{n_i}
$$
where $\sigma_i \eqdef \sum_{\ell=1}^{n_i} y_{i,\ell}$.

Notice that, the decentralized estimator turns out to be the ML estimator of
$\lambda_i$ when node $i$ can use only its own data.
However,  decentralized estimation yields reliable estimates
only when the number of samples $n_i$ is large enough. 

In our heterogenous set-up, it may happen that some nodes satisfy such a
condition, while other ones do not have enough data, thus resulting in a poor
estimation.
In this paper we propose a distributed estimation scheme in which every node,
especially the ones with fewer measurements, take advantage from cooperating
with neighboring nodes.

How the measurements at other nodes can help the local estimation at a given
node is a nontrivial issue and needs to be investigated by means of a suitable
probabilistic framework. 
Specifically, we adopt a Bayesian model in which all the unknown arrival
rates $\lambda_i$s are i.i.d. random variables ruled by a common probability
distribution that captures the spatial variability of the process intensity.

This model belongs to the family of Poisson mixtures and is broadly accepted
  in the (centralized) statistical literature for modeling non-homogeneous
  scenarios, see, e.g., \cite{freedman1962poisson,massey1996estimating}, and
  \cite[and references therein]{karlis2005mixed} for a through and extensive
  survey of properties and applications.
  As customary, see, e.g., \cite{casella}, we adopt as common distribution the
  conjugate prior of the Poisson, which is the Gamma distribution. This choice
  allows one to obtain a close-form expression for the posterior (predictive)
  distribution. 
  However, notice that, our model extends the classical Gamma-Poisson mixture
  since the sample sizes $n_i$, $i\in\until{N}$, can be different and thus
  exhibits even more flexibility.

A scheme
of the proposed Bayesian framework for our network scenario is depicted in
Figure~\ref{fig:network_scheme}.
\begin{figure}[h]
  \centering
  \includegraphics[width=0.35\textwidth]{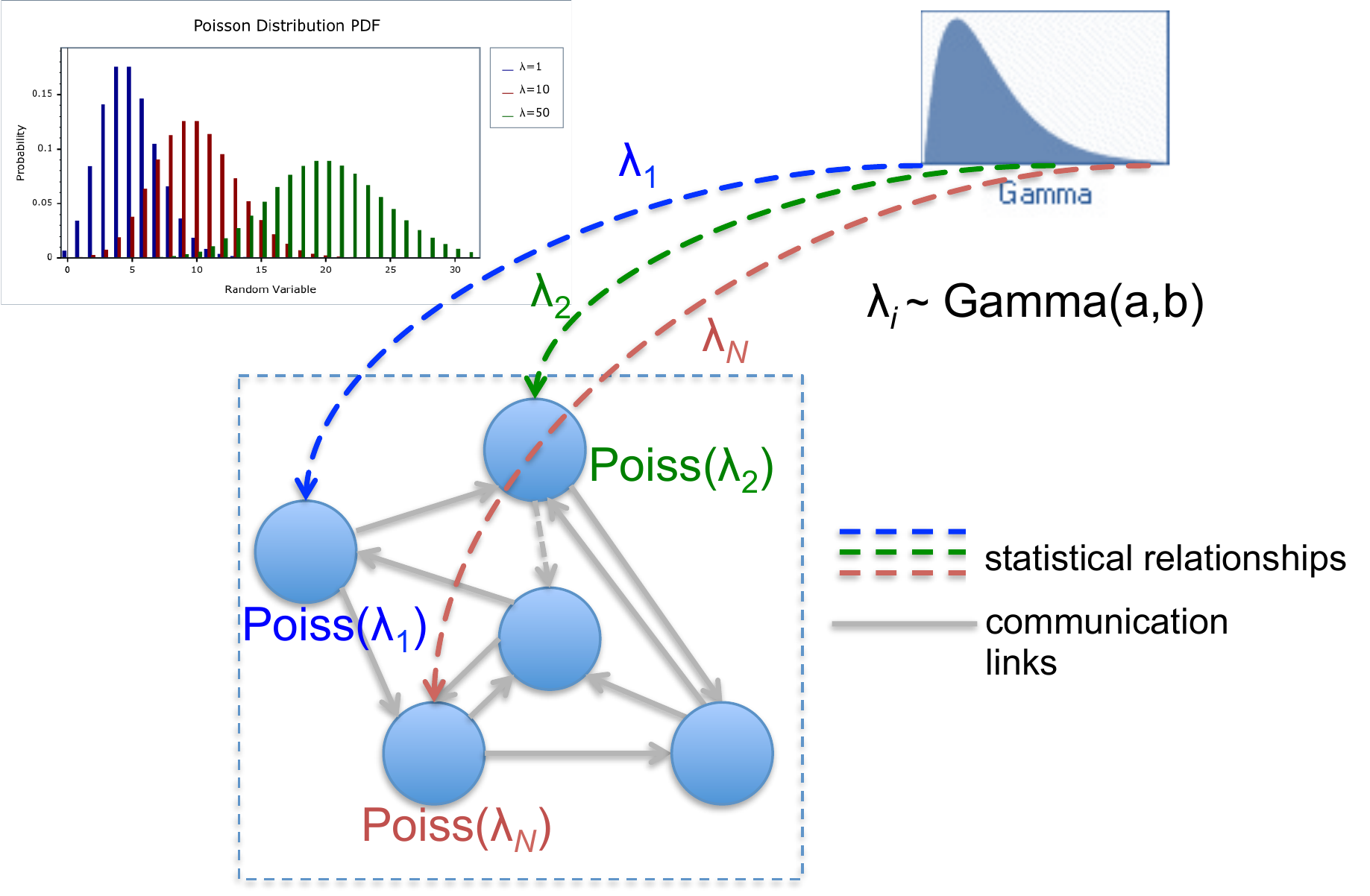}
  \caption{Scheme of the Bayesian model}
  \label{fig:network_scheme}
\end{figure}

\subsection{Empirical Bayes approach in monitoring networks}
In applying a Bayesian estimation approach to a network context, the assumption
that the prior distribution is fully known to all monitors is rather strong
and may be a severe limitation in realistic scenarios.  To overcome this
limitation we adopt the Empirical Bayes approach in which only the class of the
prior is known, 
 i.e., $\lambda_i \sim \text{Gamma}(a,b)$,
where the shape parameter $a$ is known, but the scale parameter $b$ is
unknown.
The assumption that $a$ is known, while only $b$ is unknown, says that only the
shape of the Gamma distribution (determined by the parameter $a$) is known,
while the scaling is not. This assumption is reasonable in many applications,
since it is a way to embed a rough information on the phenomenon, and is
customary for the sake of mathematical tractability, \cite{casella}.

The hyperparameter $b$ can be estimated via a ML procedure. To this aim, we need
the joint distribution of all measurements $\{y_{i,\ell} \}_{\ell=1,\ldots,n_i}$
for each agent $i$.
The likelihood function is the product of the marginal distributions of all agents
  \begin{equation}
    L(\bm{y}_1, \ldots, \bm{y}_N|b)=\prod_{i=1}^N p(\bm{y}_i |b)
    \label{eq:likelihood}
  \end{equation}
  where $\bm{y}_i=[y_{i,1}\ \cdots \ y_{i,n_i}]^T$. The marginal
  distribution of agent $i$ is derived from the joint distribution of $\bm{y}_i$
  and $\lambda_i$, 
\begin{align}
  p(\bm{y}_i|b) &=\int_0^{\infty} \left( \prod_{\ell=1}^{n_i} f(y_{i,\ell}|
    \lambda_i,b) \right) p(\lambda_i|b) \mathrm{d}\lambda_i \nonumber \\
  &= \int_0^{\infty} \frac{\lambda_i^{\sigma_i}
    e^{-n_i\lambda_i}}{\prod_{\ell=1}^{n_i} y_{i,\ell}!} \frac{\lambda_i^{a-1}
    e^{-\lambda_i/b}}{\Gamma(a)b^a} \mathrm{d}\lambda_i \nonumber\\
  &= \frac{\Gamma(\sigma_i+a)}{ \Gamma(a) b^a \prod_{\ell=1}^{n_i} y_{i,\ell}!}
  \left(\frac{b}{n_ib+1}\right)^{\sigma_i+a} \label{eq:neg_bin}
\end{align}
By using eq. \eqref{eq:neg_bin} into eq. \eqref{eq:likelihood} the likelihood is rewritten as:
\begin{equation}
L(\bm{y}_1, \ldots, \bm{y}_N|b) \propto \frac{1}{b^{aN}}  \prod_{i=1}^N \left(\frac{b}{n_ib+1}\right)^{\sigma_i+a}.  %
\label{eq:like}
\end{equation}
Thus, the ML estimator $\bML$ of $b$ can be found by solving the following
optimization problem
\begin{align}
\label{eq:b_hat}
\bML &= \argmin_{b\in\mathbb{R}_+} \!\left\{ \! aN   \log b - \! \sum_{i=1}^N
  (\sigma_i+a) \log\! \left(\frac{b}{n_ib+1}\right) \right\} \!\! \nonumber\\
         &= \argmin_{b\in\mathbb{R}_+}  \sum_{i=1}^N f(b; n_i,\sigma_i). 
\end{align}

The problem can be solved in closed-form only for the \emph{homogeneous} case
where all $n_i$s are equal, i.e., for $n_i=n/N$, recalling that $n$ is the total
number of measurements. In this case the ML estimator of $b$ based on the entire
set of measurements is given by
\begin{equation*}
  \bhom=\frac{1}{an} \sum_{i=1}^N \sigma_i  = \frac{\sigma}{an} 
\end{equation*}
where $n \eqdef \sum_{i=1}^N n_i$, and $\sigma \eqdef \sum_{i=1}^N \sigma_i$.

After obtaining an estimate for $b$, the Empirical Bayes estimator of the
arrival rate $\lambda_i$ that minimizes the Mean Square Error (MMSE) can be
obtained by computing the conditional mean of the posterior distribution
$p(\lambda_i|\bm{y}_i,b)$.  The latter is given by the ratio between the joint
pdf $p(\bm{y}_i,\lambda_i|b)$ and the marginal pdf $p(\bm{y}_i|b)$ as from
eq. \eqref{eq:neg_bin}, i.e.,
  \begin{equation}
    p(\lambda_i|\bm{y}_i,b) = \frac{\lambda_i^{\sigma_i+a-1}
      e^{-\lambda_i\frac{n_ib+1}{b}}}{\Gamma(\sigma_i+a)}
    \left(\frac{b}{n_ib+1}\right)^{-\sigma_i-a} %
    \label{eq:gamma_post}
  \end{equation}
  Eq. \eqref{eq:gamma_post} is a Gamma pdf with parameters
  $(\sigma_i+a,\frac{b}{n_ib+1})$, hence the \emph{Empirical Bayes MMSE estimator} of each
  $\lambda_i$ is
\begin{equation}
\label{eq:lambda_i}
  \lambdaEB_i = \E[\lambda_i|\bm{y}_i,\bML] = \frac{\bML}{\bML n_i+1} (a + \sigma_i).
\end{equation}
\begin{remark}
  It is worth highlighting that the Bayesian estimate is especially useful when
  $n_i$ is small.  In fact, nodes improve the quality of their local estimate by
  combining the frequentist estimation (based only on local observations) with a
  correction term (based on a prior global knowledge), which is estimated in a
  cooperative way.  Indeed, we can rewrite
  $\lambdaEB_i = \rho\frac{\sigma_i}{n_i} + (1-\rho) a \bML$, with
  $\rho = \frac{\bML n_i}{\bML n_i+1}$.
When the local information is abundant ($n_i\rightarrow\infty$ and thus
$\rho\rightarrow 1$), the MMSE estimator \eqref{eq:lambda_i} tends towards
$\lambdaDEC_i$, meaning that when $n_i$ is large no further
information can be inferred from the network.  Conversely, when local
information is scarce, i.e., the sample size $n_i$ is small, even one, the MMSE
estimator \eqref{eq:lambda_i} approaches the estimate of the global mean
$\E[\lambda_i] = a b$. 
\end{remark} 

In the following we will also consider an \emph{ad-hoc estimator} obtained by
using $\bhom$ instead of the optimal $\bML$, i.e.,
\begin{equation*}
  \lambdaAH_i \eqdef
  \frac{\bhom}{\bhom\, n_i+1} (a + \sigma_i)  =
  \frac{\sigma}{an+\sigma n_i} (a + \sigma_i).
\end{equation*}

Clearly, the estimator $\lambdaEB_i$ follows the ``rationale'' of the Empirical
Bayes approach and, therefore,
has performance guarantees inherited from the ML procedure.
Conversely, the ad-hoc estimator is an alternative that at the moment has the
only advantage of having a closed-form expression, whose performances need to be
understood. In the rest of the paper we will show that, not only this estimator
leads to a simpler and faster distributed algorithm, but also that for large
number of agents performs as $\lambdaEB_i$.

\subsection{Distributed estimators}
From eq. \eqref{eq:lambda_i} it is clear that each agent can compute the
Empirical Bayes MMSE estimator provided it knows $\bML$. Optimization problem
\eqref{eq:b_hat}, giving the ML estimator of $b$, has a separable cost (i.e.,
the sum of $N$ local costs), hence it can be solved by using available
distributed optimization algorithms for asynchronous networks
\cite{zanella2012asynchronous,nedic2013distributed}.
We propose a distributed estimator in which each node implements the local
update rule of the chosen distributed optimization algorithm.

The ML estimation problem~\eqref{eq:b_hat} is not guaranteed to be convex in
general. This is quite common in the estimation literature. 
However, since the function is coercive,
there exists (at least) a minimizer and, thus, it is reasonable to apply descent
algorithms, as the one in \cite{nedic2013distributed}, which, under suitable
conditions, guarantee convergence to a local minimizer.

The {\bf \EBestimator}\, is as follows. At each $t\in\naturalzero$, each agent $i$
stores a local state $\xi_i(t)$, an estimate $\bML_i(t)$ of $\bML$ and an
estimate $\lambdaEB_i(t)$ of $\lambdaEB_i$. The node initializes its local state
$\xi_i$ to an initial value $\xi_{i0}$ chosen according to the distributed
optimization algorithm in use, and sets $\bML_i(0) = \sigma_i/(an_i)$ (which
would be the solution of \eqref{eq:b_hat} if $i$ were the only agent). Then it
updates its estimate of $\bML$ by using the local update rule of the chosen
distributed optimization algorithm, and updates the current estimate
$\lambdaEB_i(t)$ by using \eqref{eq:lambda_i}.
The algorithm is defined formally in the following table. For each
$t\in\naturalzero$, let $\{\hat{\xi}_k(t)\}_{k\in N^{I}_i(t)}$ be the collection of
states of the in-neighbors of node $i$, \texttt{opt\_local} the local update
of the chosen distributed optimization algorithm, and $\gamma(t)$ an algorithm
parameter as, e.g., a time-varying step-size.
\begin{center}
\begin{minipage}[c]{0.85\linewidth}
\small
  \begin{algorithm}[H]
    \caption{\EBestimator}
    \label{alg:EB_estimator}
    \begin{algorithmic}
      \STATE {\bf Initialization:} $\xi_i(0) = \xi_{i0}$, $\bML_i(0) = \frac{\sigma_i}{an_i}$.\\
      \STATE {\bf Iterate:}
      \begin{align*}
        (\bML_i&(t+1),\hat{\xi}_i(t+1)) =\nonumber\\ 
        &\texttt{opt\_local}\left(\bML_i(t),\hat{\xi}_i(t),
          \{\hat{\xi}_k(t)\}_{k\in N^{I}_i(t)};\gamma(t)\right)\!,\\[1.0ex]
        \lambdaEB_i&(t+1) = \frac{\bML_i(t+1)}{\bML_i(t+1) n_i+1} (a +
        \sigma_i). \nonumber
      \end{align*}
    \end{algorithmic}
  \end{algorithm}
\end{minipage}
\end{center}
As an example, we show the \texttt{opt\_local} function for the distributed
\emph{subgradient-push} method proposed in \cite{nedic2013distributed}. To be
consistent with the notation in \cite{nedic2013distributed} we let
$\xi_i = (v_i,y_i,x_i)$, which is initialized to
$(v_i(0) ,y_i(0),x_i(0)) = (1,1,x_{i0})$, with $x_{i0}$ an arbitrary initial
value. Also, $d_k(t)$ denotes the number of out-neighbors of node $k$ at
  time $t$.
%
\begin{center}
\begin{minipage}[c]{0.6\linewidth}  
\small
{\bf function} $\xi_i(t+1) =\\ 
\phantom{aaa}\texttt{opt\_local}\left(\bML_i(t),\hat{\xi}_i(t),
  \{\hat{\xi}_k(t)\}_{k\in N^{I}_i(t)};\gamma(t)\right)\!$
\begin{align*}
  v_i(t+1) & = \sum_{k\in N_i^I(t)\cup \{i\}} \frac{x_k(t)}{d_k(t)}\\[1.2ex]
  y_i(t+1)  & = \sum_{k\in N_i^I(t)\cup \{i\}} \frac{y_k(t)}{d_k(t)}\\[1.2ex]
  \bML_i(t+1) &= \frac{v_i(t+1)}{y_i(t+1)} \\[1.2ex]
  x_i(t+1) & = v_i(t+1) - \gamma(t+1) \nabla f(\bML_i(t+1);n_i, \sigma_i)
\end{align*}
\end{minipage}
\end{center}
\noindent with $\nabla f(b;n_i, \sigma_i) = \frac{a}{b (n_i b + 1)} (a n_i b - \sigma_i)$.


If problem \eqref{eq:b_hat} has a unique minimizer, the distributed
  optimization algorithm guarantees that all nodes reach consensus on the global
  minimizer. That is,
\[
\lim_{t\rightarrow\infty} \bML_i(t) = \bML, \qquad \text{for all $i\in\until{N}$}.
\]
From the convergence properties of the chosen distributed optimization algorithm
it follows immediately that the proposed distributed estimator asymptotically
computes at each node $i$ the Empirical Bayes MMSE estimator of
$\lambda_i$. 

However, most of the available distributed optimization algorithms, as the ones
in \cite{zanella2012asynchronous,nedic2013distributed}, need the tuning of a
global parameter (we denoted it $\gamma$), and typically exhibit a
sub-exponential convergence even in static graphs. To overcome these drawbacks,
we propose an alternative distributed estimator with reduced complexity that,
although suboptimal, will be shown to perform comparably to the optimal one.

The \textbf{\AHestimator\,} is defined as follows.
For each $t\in\naturalzero$, each node $i\in\until{N}$ stores in memory two local
states $s_i(t)$ and $\eta_i(t)$, an estimate $\bhom_i(t)$ of
$\bhom$, and an estimate $\lambdaAH_i(t)$ of
$\lambdaAH_i$. Let $w_{ik}(t)\in\real_{\geq0}$
be a set of weights such that $w_{ik}(t)>0$ if $(i,k)\in E(t)$ or $k=i$, and
$w_{ik}(t)=0$ otherwise. The \AHestimator\ is given in the following table. 
\begin{center}
\begin{minipage}[c]{0.85\linewidth}
  \begin{algorithm}[H]
    \caption{\AHestimator}
    \label{alg:ad-hoc_estimator}
    \begin{algorithmic}
      \STATE {\bf Initialization:} $s_i(0) = \sigma_i$, $\eta_i(0) = n_i$, $\bhom_i(0) = \sigma_i/(an_i)$, $\lambdaAH_i(0) =
      \frac{\bhom_i(0)}{\bhom_i(0) n_i+1} (a + \sigma_i)$.\\
      \STATE {\bf Iterate:}
      \begin{align}
        s_i(t+1)       &= \sum_{k\in N_i^I(t)\cup \{i\}} w_{ik}(t) s_k(t) \nonumber\\[1.0ex]
        \eta_i(t+1)     &= \sum_{k\in N_i^I(t) \cup \{i\}} w_{ik}(t) \eta_k(t)
        \nonumber\\
        \bhom_i(t+1) &= \frac{1}{a}\,\frac{s_i(t+1)}{\eta_i(t+1)} \nonumber\\[1.0ex]
        \lambdaAH_i(t+1) &= \frac{\bhom_i(t+1)}{\bhom_i(t+1) n_i+1} (a +
        \sigma_i). \nonumber
      \end{align}
    \end{algorithmic}
  \end{algorithm}
\end{minipage}
\end{center}

\medskip

We can rewrite the update of $\eta_i(t)$ and $s_i(t)$ by using an aggregate
dynamics. That is, let $\eta(t) = [\eta_1(t) \ldots \eta_N(t)]^T$ and $s(t) =
[s_1(t) \ldots s_N(t)]^T$ be the aggregate states, their dynamics is given by
\begin{equation}
\begin{split}
  s(t+1) &= W(t) s(t)\\
  \eta(t+1) &= W(t) \eta(t)
\end{split}
\label{eq:s_eta_dyn}
\end{equation}
with $s(0) = [\sigma_1 \ldots \sigma_N]^T \eqdef \sigmavec$, $\eta(0) = [n_1 \ldots
n_N]^T \eqdef \nvec$ and $W(t)$ the matrix with elements $w_{ij}(t)$.
Let us denote 
\begin{equation}
\Phi(t) = W(t-1) \cdots W(0) 
\label{eq:state_trans_matrix}
\end{equation}
the state transition matrix associated to each one of the linear systems
\eqref{eq:s_eta_dyn}, so that 
\begin{equation}
  \begin{split}
      s(t) &= \Phi(t) s(0)\\
     \eta(t) &= \Phi(t)\eta(0).
     \label{eq:s_eta_evol}
   \end{split}
\end{equation}
For the algorithm to converge we need the following assumption together with the
Assumption~\ref{assum:graph} (uniform joint connectivity of the communication digraph).

\begin{assumption}[Properties of $W(t)$] %
\label{assum:col_stochastic}
  For each $t\in\naturalzero$, the matrix $W(t)$ is column stochastic, i.e.,
  $\sum_{i=1}^n w_{ik}(t) = 1$, and there exists a positive constant $\alpha>0$
  such that $w_{ii}(t)>\alpha$ and $w_{ik}(t) \in \{0\} \union [\alpha,1]$.
\end{assumption}

\begin{remark}
  It is worth noting that the column stochasticity assumption above is not the
  usual assumption used in linear consensus algorithms in which row
  stochasticity is assumed.
\end{remark}

\begin{lemma}%
\label{lem:weak_ergodicity}
  Let $t\mapsto G(t)$ be a uniformly jointly strongly connected graph
  (Assumption~\ref{assum:graph}) and $\{W(t)\}_{t\geq0}$ a sequence of matrices
  satisfying Assumption~\ref{assum:col_stochastic}. Then denoting $\Phi(t) =
  W(t) W(t-1) \cdots W(0)$ with $(i,k)$ element $\phi_{ik}(t)$, the following holds true.
  \begin{enumerate}
  \item The matrix sequence $\{W(t)\}_{t\geq0}$ is weakly ergodic, i.e., 
    \[
    \lim_{t\rightarrow\infty} | \phi_{ik}(t) - \phi_{ih}(t) | = 0 
    \]
    for all $i,k,h \in \until{N}$.
  \item There exists some $\mu>0$ such that for all $t\geq 0$,
    $\textbf{1}^T\Phi(t) \geq \mu>0$, i.e., for any $i\in\until{N}$,
    \[
    \sum_{k=1}^N \phi_{ik}(t) > \mu.
    \]
  \end{enumerate}
\end{lemma}

The result is well-known and can be found, e.g., in
\cite{nedic2013distributed}. Further references on this result under the same or
different connectivity assumptions are
\cite{tsitsiklis1986distributed,seneta2006non,jadbabaie2003coordination,benezit2010weighted}.

\medskip
\begin{proposition}
  \label{prop:ad_hoc_convergence}
  Let Assumption~\ref{assum:graph} hold. Then the \AHestimator\,
  (Algorithm~\ref{alg:ad-hoc_estimator}) satisfies
  \[
  \begin{split}
    \lim_{t\rightarrow\infty} \bhom_i(t) &= \bhom\\
    \lim_{t\rightarrow\infty} \lambdaAH_i(t) &= \lambdaAH_i.
  \end{split}
  \]
\end{proposition}
\begin{proof}
  From Lemma~\ref{lem:weak_ergodicity}, $\{W(t)\}_{t\geq0}$ is weakly ergodic so
  that there exists a sequence of stochastic vectors $\bar{\phi}(t) =
  [\bar{\phi}_1(t) \ldots \bar{\phi}_N(t)]$ such that for each $i\in\until{N}$,
  $\phi_{ik}(t) \rightarrow \bar{\phi}_i(t)$ for all $k\in\until{N}$. From
  eq. \eqref{eq:s_eta_evol} it follows that 
\[
s_i(t) \rightarrow \bar{\phi}_i(t)
  \sum_{k=1}^N s_k(0) = \bar{\phi}_i(t) \sum_{k=1}^N \sigma_k(0)
\]
 and 
\[
\eta_i(t)
  \rightarrow \bar{\phi}_i(t) \sum_{k=1}^N \eta_k(0) = \bar{\phi}_i(t) \sum_{k=1}^N
  n_k(0).
\]
 Thus, $\frac{s_i(t)}{\eta_i(t)} \rightarrow \frac{\sum_{k=1}^N \sigma_k(0)}{
  \sum_{k=1}^N n_k(0)}$, so that the proof follows.
\end{proof}

\begin{remark}[Connectivity assumption]
  The result of Proposition~\ref{prop:ad_hoc_convergence} can be proven also if
  Assumption~\ref{assum:graph} is replaced by Assumption~1 in
  \cite{benezit2010weighted}, that is if $\{W(t)\}_{t\geq0}>0$ is a
    stationary and ergodic sequence of stochastic matrices with positive
    diagonals, and $\E[W]$ is irreducible. The result can be proven by following
    the same line of proof developed therein. Thus, the ad-hoc estimator could
  be implemented also in a monitoring network with stochastic gossip
  communication.
\end{remark}

\begin{remark}
  The proposed algorithms are based on the assumption that the arrival-rates are
  constant in the observation window. Therefore, one can apply the algorithms
  iteratively by recomputing the estimates on different windows of data, thus
  getting time-varying arrival-rates.
\end{remark}

To conclude this section, we point out that the update of the subgradient push
optimization algorithm includes a push-sum consensus step, i.e., a diffusive
update based on a column stochastic matrix (with coefficients
$w_{ik}(t) = 1/d_k(t)$ for $k\in N_i(t)$), which is the same used in the
\AHestimator\, (Algorithm~\ref{alg:ad-hoc_estimator}).
However, in the subgradient push this update is part of a gradient
descent step. In fact, the role and the evolution of the involved variables,
i.e., $(s_i,\eta_i)$ and $(v_i,y_i)$ respectively, are different as well as the
convergence rates. Indeed, the ad-hoc estimator exhibits the exponential
convergence of linear consensus protocols as opposed to the much slower
$O(\ln t/\sqrt{t})$ rate of the subgradient-push \cite{nedic2013distributed}.

These considerations, together with the lower computational burden, make the
\AHestimator\, appealing even though not optimal. In the following section we
will show that it actually performs very closely to the MMSE estimator.

\section{Estimator performance analysis} %
\label{sec:theoretical_analysis} %
In this section we analyze the performance of the proposed distributed
estimators. In particular, for the \EBestimator, lacking a closed form for the
update rule of $\bML$ and, in turn, of $\lambdaEB$, we are able to derive only
steady-state ($t\rightarrow \infty$) and asymptotic ($N\rightarrow\infty$)
bounds.
Conversely, for the \AHestimator\, a transient analysis (at any
$t\in\naturalzero$) can be derived for any $N\in\natural$. 

To develop the performance analysis, we consider a ``special'' agent, we label
it as $j$, that does not participate to the computation of $\bML$ (respectively
$\bhom$). Under this assumption, it turns out that $\sigma_j$ is independent of
$\bML$ (respectively $\bhom$). Clearly, $\sigma_j$ is also independent of any
local estimate 
$\bML_i(t)$ (respectively $\bhom_i(t)$), $i\in\until{N}$, at any $t\in\naturalzero$. Consistently, here we
can simply assume that for the computation of $\lambdaEB_j(t)$ (respectively
$\lambdaAH_j(t)$), agent $j$ uses the local estimate of $b$ computed by one of
its in-neighbors, that is, e.g., $\bhom_j(t) = \bhom_i(t)$ for some
$i\in N_j^I(t)\subset\until{N}$.

Notice that, in practice, the analysis developed for such an agent $j$ holds
approximately for any node in the network participating to the distributed
computation. Indeed, due to the large number of agents $N$, for any agent $i$
(running the distributed algorithm), $\bML$ (respectively $\bhom$) and
$\sigma_i$ are very weakly correlated. The validity of this statement will be
corroborated by the Monte Carlo analysis in the next section.

We start by deriving the Cramer-Rao lower bound (CRB) for any unbiased estimator
$\hat{b}$ of $b$, i.e.,
\[
\VAR[\hat{b}] \geq \mathrm{CRB}(b).
\]
\begin{lemma}
The CRB for the estimation of the hyperparameter $b$ is given by
\begin{equation}
  \mathrm{CRB}(b) = \frac{b/a}{\sum_{i=1}^N
    \frac{n_i}{n_ib+1}}.
\label{eq:CRB}
\end{equation}
\label{lem:CRB}
\end{lemma}
The proof is reported in Appendix~\ref{sec:lem_CRB}.

\subsection{Analysis of the \EBestimator}
For the \EBestimator\, we can analyze the performance only once consensus on the
optimal value $\bML$, and thus on $\lambdaEB_j$, has been reached. Specifically,
let us recall that
\begin{equation}
\label{eq:lambdaEB_0}
  \lambdaEB_j = \frac{\bML}{\bML n_j+1} (a + \sigma_j).
\end{equation}

Despite no analytical expression is available for the ML estimator $\bML$ (nor
for its moments) the asymptotic analysis ($N\rightarrow\infty$) of the
corresponding MMSE estimator $\lambdaEB_j$ can be obtained by using the
properties of ML estimation.
In particular, any ML estimator is asymptotically unbiased and efficient,
hence $\E[\bML]\rightarrow b$ and $\VAR[\bML]\rightarrow \mathrm{CRB}(b)$ as $N\rightarrow\infty$.
From equation \eqref{eq:CRB}, 
\begin{align}
  \mathrm{CRB}(b) &= \frac{b/a}{\sum_{i=1}^N
    \frac{n_i}{n_ib+1}}\nonumber\\ 
&\leq \frac{b/a}{\sum_{i=1}^N
    \frac{1}{\nmax b+1}} = \frac{b/a}{N
    \frac{1}{\nmax b+1}} \stackrel{N\rightarrow\infty}{
    \xrightarrow{\hspace*{1cm}}} 0 \,,  
\label{eq:CRB_asympt}
\end{align}
so that $\VAR[\bML]\rightarrow 0$ as $N\rightarrow\infty$.

These results on the asymptotic properties of $\bML$ can be used to prove the following theorem
characterizing the asymptotic behavior of the \EBestimator\,.

Due to the nonlinear dependency of $\lambdaEB$ from $\bML$, we will perform an
approximate analysis by considering the Taylor expansion of $\lambdaEB$ around
$\E[\bML]$. For tractability we will consider respectively the second-order and
the first-order approximations for mean and variance.
The numerical analysis in Section~\ref{sec:numerical_analysis} will confirm the
validity of such an approximation.
\begin{theorem}
\label{thm:mean_var_EB_estimator}
Consider a network of monitors as in Section~\ref{sec:prob_form} running the
\EBestimator. Then, as $N\rightarrow \infty$, it holds
true
\begin{equation}
\E[\lambdaEB_j|\lambda_j] \longrightarrow
\frac{b}{1 + n_j b}(a+n_j\lambda_j).
\label{eq:E_lambdaEB_lim_N}
\end{equation}
to second order and 
\begin{equation}
  \VAR[\lambdaEB_j| \lambda_j] \longrightarrow \left( \frac{b}{1+n_jb}
  \right)^2 n_j\lambda_j,
\label{eq:VAR_lambdaEB_lim_N}
\end{equation} 
to first order.
\end{theorem}

\subsection{Analysis of the \AHestimator}
The closed-form update of the \AHestimator\, allows us to perform a more
detailed analysis. In particular, we are able to characterize mean and variance
of the local estimator $\lambdaAH_j(t)$ during the algorithm evolution
(transient analysis) and for any fixed value of the number of nodes $N$.

As we have done for the \EBestimator, also for the ad-hoc one we first
characterize the estimator of the hyperparameter $b$.

Before addressing the transient analysis, we compute mean and variance of the
consensus value $\bhom$.
\begin{proposition}
  The homogeneous estimator of $b$, $\bhom$, is unbiased, i.e., 
\[
\E[\bhom]=b,
\] 
and has variance
\begin{equation}
  \VAR[\bhom] = \frac{b}{a n} + \frac{b^2}{a n^2} \sum_{i=1}^{N} n_i^2 \,.
  \label{eq:VAR_bhom}
\end{equation}
Moreover, the estimator $\bhom$ is consistent, i.e., it converges in probability to the
true value $b$ as $N\rightarrow \infty$.
\label{prop:VAR_bhom}
\end{proposition}
The proof is given in Appendix~\ref{sec:prop_VAR_bhom}.

\begin{remark}
  The estimator $\bhom$ is the ML estimator in the homogenous case, i.e., when
  $n_i = n/N$ for all $i\in\until{N}$, and it attains the Cramer-Rao Bound (CRB)
  not only asymptotically, but for any $N$.
  It follows by substituting $n_i = n/N$ in \eqref{eq:CRB} and
  \eqref{eq:VAR_bhom}. 
  Interestingly, we will show in the numerical analysis that even in the
  non-homogeneous scenario the ML estimator $\bML$ approaches the CRB at any
  fixed $N$. 
\end{remark}

The following lemma gives a characterization of the transient local estimates.
\begin{lemma}
  \label{lem:E_VAR_bhom_t}
  Let Assumption~\ref{assum:graph} and Assumption~\ref{assum:col_stochastic}
  hold. For all $i\in\until{N}$, the local estimator $\bhom_i(t)$ used in the
  \AHestimator\, (Algorithm~\ref{alg:ad-hoc_estimator}) is unbiased at any
  $t\in\naturalzero$, i.e.,
  \[
  \E[\bhom_i(t)]=b,
  \] 
  and has variance
  \begin{equation}
      \label{eq:VAR_bhom_t}
       \begin{split}
         \VAR[\bhom_i(t)] = \frac{b}{a} \frac{\sum_{k=1}^N \phi_{ik}(t)^2
           n_k}{{\left(\sum_{k=1}^N \phi_{ik}(t) n_k\right)^2}} + \frac{b^2}{a}
         \frac{ \sum_{k=1}^N \phi_{ik}(t)^2 n_k^2}{\left(\sum_{k=1}^N
             \phi_{ik}(t) n_k\right)^2}
       \end{split}
     \end{equation}
     where $\phi_{ik}(t)$ is the element $(i,k)$ of the state transition matrix
     $\Phi(t)$ defined in \eqref{eq:state_trans_matrix}.  

     Moreover, as $t\rightarrow\infty$, the variance of $\bhom_i(t)$ converges
     exponentially to the variance of $\bhom$, \eqref{eq:VAR_bhom},
     and 
     satisfies
     \begin{equation}
       |\VAR[\bhom_i(t)]-\VAR[\bhom]| \leq \frac{b(1+2b\nmax)}{\mu a}
       \delta(\Phi(t)), 
       \label{eq:bound_VAR_bhom_t}
     \end{equation}
     where $\mu>0$ is given in Lemma~\ref{lem:weak_ergodicity} and
     $\delta(\Phi(t))$ is a (proper) coefficient of ergodicity\footnote{A
       coefficient of ergodicity is a function $\tau(\cdot)$ continuous on the
       set of raw (respectively column) stochastic matrices satisfying $0\leq
       \tau(A) \leq 1$. It is proper if $\tau(A) = 0$ if and only if $A =
       \textbf{1} v^T$, with $\textbf{1} = [1 \; \ldots 1]^T$ and $v$ a
       stochastic vector, \cite{seneta2006non,vaidya2011distributed}.}
     (exponentially decaying with time) defined as $\displaystyle\delta(A) \eqdef
     \max_{k}\max_{i_1,i_1} |a_{i_1,k} - a_{i_2,k} |$.
\end{lemma}
The proof is reported in Appendix~\ref{sec:lem_E_VAR_bhom_t}.

\begin{remark}
  It is worth noting that $\mu$ and $\delta(\Phi(t))$ in
  equation~\eqref{eq:bound_VAR_bhom_t} typically depend on $N$ (and thus also on
  $n$). The interesting aspect of the given result is that the bound provided in
  the previous lemma relates the convergence rate of $\VAR[\bhom]$ to parameters
  of the communication graph and the diffusion protocol as $\mu$ and
  $\delta(\Phi(t))$.  For example, if $\union_{\tau=tQ}^{(t+1)Q-1} G(\tau)$ is
  balanced, the matrix $\Phi(t)$ is doubly stochastic, hence $\mu=1$.
\end{remark}

\bigskip

With this characterization of the hyperparameter estimator, we are ready to
analyze the estimator $\lambdaAH_j(t)$.
Recalling the assumption that the measurements of agent $j$ do not contribute to
the computation of $\bhom_j(t)$, we clearly have\footnote{As stated before,
  agent $j$ uses as $\bhom_j(t)$ the estimate $\bhom_i(t)$ of some neighbor,
  $i\in N_j^I\subset\until{N}$ participating to the distributed computation.}
\begin{align}
\label{eq:cond_mean_approx}
  \E[\bhom_j(t)| \lambda_j] = \E[\bhom_j(t)] = b.
\end{align}
and
\begin{align*}
  \VAR[\bhom_j(t)|\lambda_j] =   \VAR[\bhom_j(t)]. 
\end{align*}

\begin{remark}
  It is worth noticing once more that although for a tractable, rigorous
  analysis we need the assumption that agent $j$ does not contribute to the
  distributed computation, in practice the results hold with good approximation
  also in the scenario in which agent $j$ contributes to the distributed
  computation. This is due to the weak impact of single measurements onto the
  aggregate quantities. In fact, for example, following analogous calculations
  as in the proof of Lemma~\ref{lem:E_VAR_bhom_t}, the conditional mean in this
  latter case turns out to be
\begin{equation}
\label{eq:cond_mean_exact}
   \E[\bhom_j(t)| \lambda_j] = b - \frac{\phi_{jj}(t) n_j}{\eta_j(t)} \left(b
     - \frac{\lambda_j}{a}\right),
 \end{equation}
 where $\phi_{jj}(t)$ is, again, the element $(j,j)$ of $\Phi(t)$ and we have
also used that, conditioned to the $\lambda_j$ of agent $j$,
\begin{equation*}
\E[y_{i,\ell}|\lambda_j] =\left\{ \begin{array}{ll} \lambda_j & i=j\\ ab
      & i\neq j\end{array} \right. .
\end{equation*}
Since as $t\rightarrow \infty$, $\frac{\phi_{jj}(t) n_j}{\eta_j(t)}
\rightarrow \frac{n_j}{n} \ll 1$, then the difference between equations
\eqref{eq:cond_mean_exact} and \eqref{eq:cond_mean_approx} is practically
negligible.
\end{remark}

In the next theorem we provide explicit transient expressions for the
conditional mean and variance of $\lambdaAH_j(t)$. Again, we use second-order
and first-order approximations respectively.
\begin{theorem}
\label{thm:mean_var_AD-HOC_estimator}
  Consider a network of monitors as in Section~\ref{sec:prob_form} running the
  \AHestimator\, (Algorithm~\ref{alg:ad-hoc_estimator}). Then the transient
  conditional mean and variance of $\lambdaAH_j(t)$ are given by
\begin{align}
\;\;\;  \E[&\lambdaAH_j(t)|\lambda_j] \approx \nonumber\\
&(a+n_j\lambda_j)\left[ \frac{b}{1+n_j b}
    - \frac{n_j}{(1+n_j b)^3} \VAR[\bhom_j(t)]
  \right] 
\label{eq:E_cond_approx_AH}
\end{align}
and
\begin{align}
\VAR[&\lambdaAH_j(t)|\lambda_j] \approx \nonumber\\ 
& \left[ \frac{b}{1+n_jb} -
  \!\frac{n_j}{(1+n_jb)^3} \VAR[\bhom_j(t)] \right]^2 n_j\lambda_j \nonumber\\
 & \;\; +\left[\frac{2n_j}{(1+n_jb)^3}\right]^2 \VAR[\bhom_j(t)]\left[
   (a+n_j\lambda_j)^2+n_j\lambda_j\right]
\label{eq:VAR_cond_approx_AH}
\end{align}
where $\approx$ indicates respectively the second-order and the first-order
approximations,
and $\VAR[\bhom_j(t)]$ is given in \eqref{eq:VAR_bhom_t}.\\[-0.8em]

Moreover, the asymptotic value $\lambdaAH_j$ satisfies, as
$N\rightarrow \infty$,
\begin{align}
  \E[\lambdaAH_j|\lambda_j] & \longrightarrow
  \frac{b}{1 + n_j b}(a+n_j\lambda_j)
\label{eq:E_cond_approx_AH_asympt}
\end{align}
to second order and
\begin{align}
\VAR[\lambdaAH_j|\lambda_j] & \longrightarrow 
\left(\frac{b}{1+n_j b} \right)^2 n_j\lambda_j
\label{eq:VAR_cond_approx_AH_asympt}
\end{align}
to first order.
\end{theorem}

\begin{remark}
  By comparing equations
  \eqref{eq:E_cond_approx_AH_asympt}-\eqref{eq:VAR_cond_approx_AH_asympt} with
  \eqref{eq:E_lambdaEB_lim_N}-\eqref{eq:VAR_lambdaEB_lim_N}, one can notice that
  at steady-state ($t\rightarrow\infty$) and for large number of agents
  ($N\rightarrow\infty$) the mean and variance of the \AHestimator\, approach
  the ones of the optimal Empirical Bayes.  Moreover, the right-hand side of
  \eqref{eq:VAR_cond_approx_AH_asympt} can be rewritten as
  $\frac{\lambda_j}{n_j} \left(\frac{n_j b}{1+n_j b} \right)^2$, which is
  clearly smaller than the variance of the decentralized estimator
  $\VAR[\lambdaDEC_j| \lambda_j] = \frac{\lambda_j}{n_j}$.
\end{remark}

\section{Numerical performance analysis} %
\label{sec:numerical_analysis}
In this section we analyze the performance of the proposed estimators. Starting
from the theoretical characterization developed in the previous section, we
perform a Monte Carlo analysis confirming the theoretical bounds and adding
other insights on the performance of the estimators.

As performance metric we adopt the Root Mean Square Error (RMSE), thus taking
into account both bias and variance of the estimators. We recall that for an estimator
$\hat\omega$ of a parameter $\omega$, the RMSE is defined as
\begin{align}
\label{eq:RMSE_def}
  \RMSE[\hat\omega] &= \sqrt{\E[(\hat\omega - \omega)^2]} \nonumber\\
  &= \sqrt{\VAR[\hat\omega] + \E^2[\hat\omega] - 2\omega\E[\hat\omega] +
    \omega^2 }.
\end{align}
Clearly, if the estimator is unbiased the RMSE coincides with the standard
deviation, i.e., $\RMSE[\hat\omega] = \sqrt{\VAR[\hat\omega]}$. 

The statistical RMSE, $\RMSE[\omega]$, will be compared with the sample value obtained through the
Monte Carlo trials, $\RMSE_{MC}$, computed as
\[
  \RMSE_{MC} = \sqrt{\frac{1}{M} \sum_{m = 1}^M (\hat\omega[m] - \omega)^2},
\]
where $M$ is the number of trials and $\hat\omega[m]$ is the $m$th estimate of
$\omega$. 

In the following we set $M = 5\times 10^4$ and, to generate the random values,
we use a Gamma distribution with parameters $a = 10$ and $b = 1$, which gives
values of $\lambda_i$ in the range $[1,25]$ with $99.98\%$ probability.

In order to challenge the \AHestimator\ we focus on a strongly inhomogeneous
network scenario. That is, we consider a network in which half of the nodes have
the maximum number of measurements in the observation window, $n_i=\nmax$ (we
set $\nmax=50$), and the remaining ones only one measurement, $n_i=1$.

We start by analyzing the transient performance of the \AHestimator. 
Consistently with the theoretical analysis, we first focus on the time evolution
of the RMSE of $\bhom_i(t)$. Notice that, since $\bhom_i(t)$ is an unbiased
estimator it holds $\RMSE[\bhom_i(t)] = \sqrt{\VAR[\bhom_i(t)]}$. We compare the
evolution of the sample RMSE with the theoretical expression obtained from
\eqref{eq:VAR_bhom_t}.
For this analysis we set $N=20$ and consider two possible communication models. In the
first one we consider a fixed, directed and very sparse communication graph
defined as follows: starting from a directed cycle, we add the edges $(3,1)$,
$(3,2)$, $(4,1)$ and $(4,2)$ to unbalance the graph.
In Figure~\ref{fig:b_transient_fixed} we plot the
evolution of $\RMSE[\bhom_i(t)]$ for $4$ nodes, two of them with $\nmax$
measurements and two with one measurement. The theoretical curve predicts very
accurately the sample RMSE obtained by the Monte Carlo trials, as highlighted in
the inset. As expected, the RMSE of the different nodes converges to the
consensus value $\RMSE[\bhom]$ obtained from \eqref{eq:VAR_bhom}.

As a second scenario we challenge the algorithm on a time-varying
topology. Namely, we consider a graph obtained by extracting at each
time-instant an Erd\H{o}s-R\'enyi graph with parameter $0.01$. We choose a small
value, so that at a given instant the graph is disconnected with high
probability, but it turns out to be uniformly jointly connected with $Q=36$.

In Figure~\ref{fig:b_transient_erdos} we again compare the theoretical
evolutions of $\RMSE[\bhom_i(t)]$ with their sample counterparts. We can
highlight two main differences with respect to the previous scenario. The curves
have some constant portions showing that nodes can be isolated for some
time-intervals. However, the convergence is faster compared to the fixed
scenario. This can be explained by the higher density of the union graph in the
time-varying scenario as opposed to the sparsity of the fixed graph. In fact, we
noticed that increasing the Erd\H{o}s-R\'enyi graph parameter increases the
convergence speed.

\begin{figure}
\centering
\includegraphics[width=0.4\textwidth]{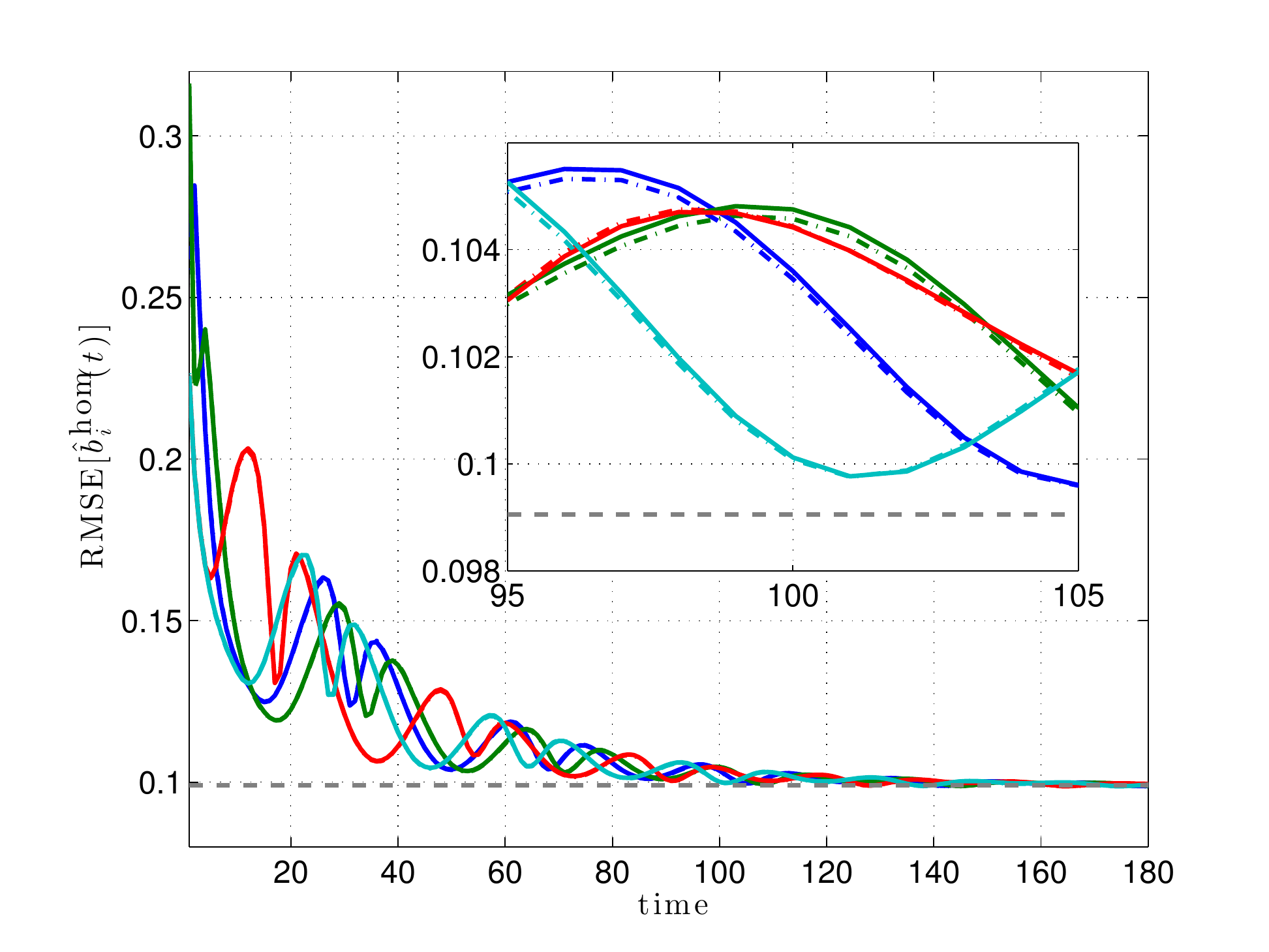}
\caption{Time evolution of $\RMSE[\bhom_i(t)]$ under fixed communication
  graph. The solid lines indicate the theoretical expressions, while the
  dash-dot lines are the ones obtained via Monte Carlo simulations. The dashed
  horizontal line is the theoretical consensus value.}%
\label{fig:b_transient_fixed}
\end{figure}

\begin{figure}
\centering
\includegraphics[width=0.4\textwidth]{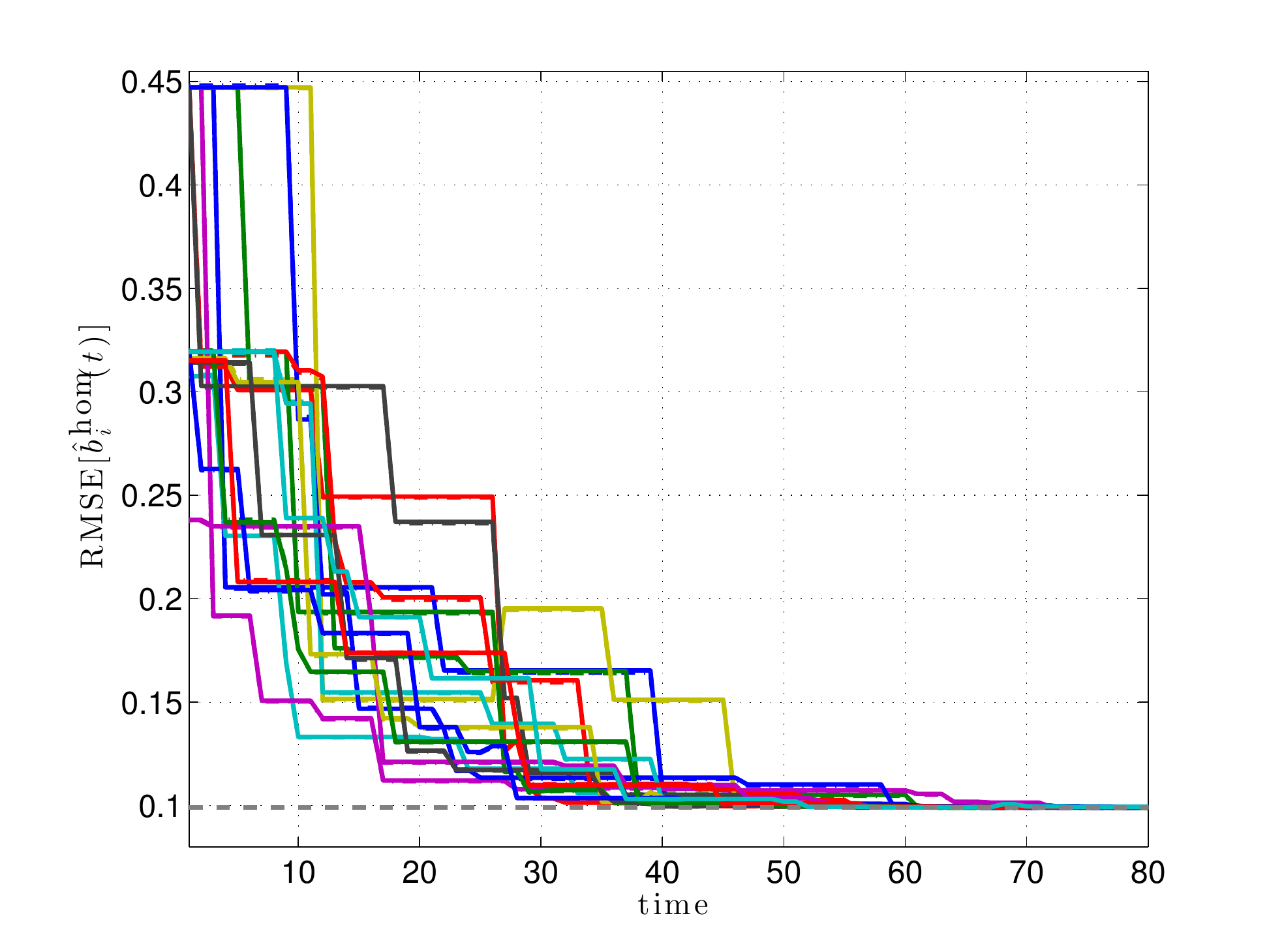}
\caption{Time evolution of $\RMSE[\bhom_i(t)]$ under time-varying communication
  graph. The solid lines indicate the theoretical expressions, while the
  dash-dot lines are the ones obtained via Monte Carlo simulations. The dashed
  horizontal line is the theoretical consensus value.}
\label{fig:b_transient_erdos}
\end{figure}

\begin{figure}
\centering
\includegraphics[width=0.4\textwidth]{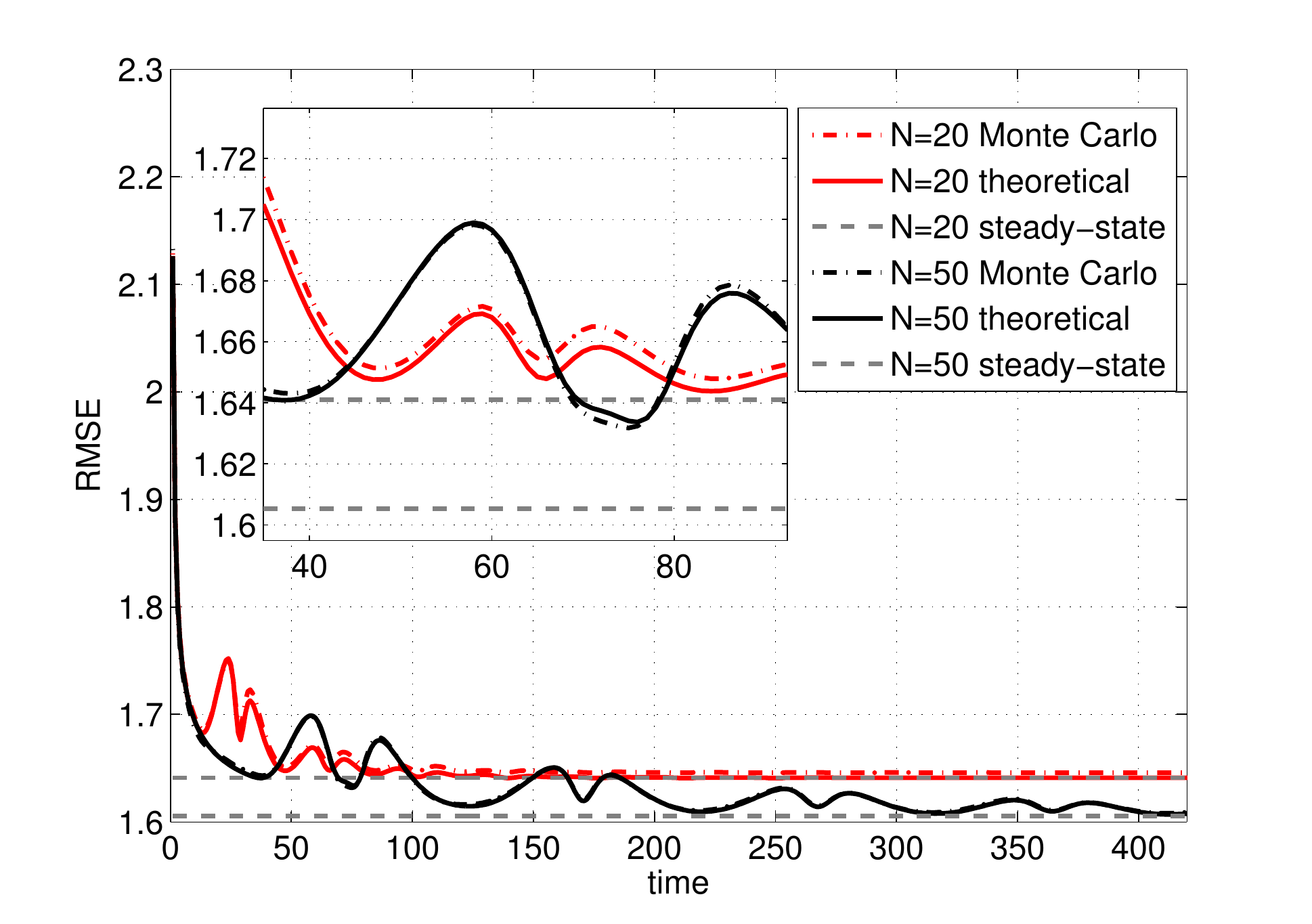}
\caption{Time evolution of $\RMSE[\lambdaAH_i(t)]$ under fixed communication
  graph for $N=20$ (red lines) and $N=50$ (black lines). The solid lines
  indicate the theoretical expressions, while the dash-dot lines are the ones
  obtained via Monte Carlo simulations. The dashed horizontal lines are the
  theoretical steady-state values. }\label{fig:lambda_transient_fixed}
\end{figure}

Now we focus on the transient behavior of $\RMSE[\lambdaAH_i(t)]$ for a generic
node $i$. Notice that we use the index $i$, rather than $j$, meaning that in the
Monte Carlo trials the node under investigation participates to the computation
of $\bhom_i(t)$. This allows us to show that the uncorrelation assumption in
Section~\ref{sec:theoretical_analysis} is in fact reasonable.
Moreover, the computations will also confirm the validity of the low-order
approximation made to derive the theoretical expressions of mean and variance.

In Figure~\ref{fig:lambda_transient_fixed} we compare the theoretical and sample
evolution of $\RMSE[\lambdaAH_i(t)]$. The theoretical curve is obtained by
plugging equations \eqref{eq:E_cond_approx_AH} and \eqref{eq:VAR_cond_approx_AH}
into equation \eqref{eq:RMSE_def}. 
We compute the curves for two different values of $N$, namely $N=20$ and
$N=50$. The difference between the theoretical and sample curves is already
minimal for $N=20$ (showing a very weak correlation between $\bhom_i(t)$ and
$\sigma_i$) and completely disappears for $N=50$ (showing that the
correlation has practically no more influence).
We want to stress that running the same computation for a node not participating
to the distributed computation (hence matching the uncorrelation assumption) the
theoretical and sample curves are indistinguishable even for $N=20$. This
suggests that the low-order approximation does not affect the goodness of the
prediction.

It is worth noting that by increasing $N$ the steady-state value,
$\RMSE[\lambdaAH_i]$, decreases, since the hyperparameter $b$ is estimated by
means of a larger sample. This aspect will be better highlighted in the
following asymptotic analysis in which we focus on how $\RMSE[\lambdaAH_i]$
varies with $N$.

In the asymptotic analysis we consider both the \EBestimator\ and the
\AHestimator\, by comparing again the predicted theoretical values with the
sample counterparts. 
As in the transient analysis we first focus on the estimation of $b$. 

In Figure~\ref{fig:b_asympt_N} we plot the sample RMSE of the two estimators (ML
and homogeneous) and compare them with the theoretical value of the homogeneous
estimator and with the Cramer-Rao Bound (CRB).
As expected, for each fixed $N$ the RMSE of the ML estimator is lower than the
homogeneous one due to the (strong) inhomogeneity of the network. Once again, we
recall that the homogeneous estimator coincides with the ML estimator only when
the network is homogeneous (i.e., $n_i = n/N$ for any $i\in\until{N}$). As
already experienced in the transient analysis, the theoretical values of
$\RMSE[\bhom]$ practically coincide with the sample ones. Although for the ML
estimator we have no theoretical expression for fixed $N$, the picture shows a
very interesting property. That is, the ML estimator achieves the CRB not only
asymptotically ($N\rightarrow\infty$) as predicted by the theory, but also for
each fixed $N$.
Interestingly, in accordance to the theoretical results in the previous section,
also the homogeneous estimator achieves the CRB as $N$ goes to infinity.   

Finally, we analyze the RMSE of the estimators of $\lambda_i$. We consider an
agent with one measurement ($n_i = 1$) and use the most frequent value for the
arrival rate, i.e., the mode of the Gamma distribution, $\lambda_i = (a-1) b$.

In Figure~\ref{fig:lambda_asympt_N} we plot the sample RMSE of the Empirical
Bayes and ad-hoc estimators. We also plot the (sample) values of the
decentralized estimator. We decided to normalize all the curves to the
theoretical value of the decentralized estimator in order to highlight the
improvements of the proposed distributed estimators. Clearly, the sample values
of the decentralized estimator are approximately equal to one, with minor
fluctuations only due to the finite number of samples.
We compare the sample curves with the theoretical curve of the
homogeneous estimator and with the theoretical asymptotic value as
$N\rightarrow\infty$. 
The theoretical curve is obtained as follows: $\E[\lambdaAH_i]$ and
$\VAR[\lambdaAH_i]$ can be computed by plugging $\VAR[\bhom]$ from
\eqref{eq:VAR_bhom} in \eqref{eq:E_cond_approx_AH} and
\eqref{eq:VAR_cond_approx_AH}; then, $\RMSE[\lambdaAH_i]$ is obtained by
plugging $\E[\lambdaAH_i]$ and $\VAR[\lambdaAH_i]$ into
\eqref{eq:RMSE_def}.

Again, although computed under the uncorrelation assumption and neglecting
higher order terms, the theoretical expression $\RMSE[\lambdaAH_i]$ predicts
very accurately the sample values (cross markers and solid curve).

The plot confirms how the distributed estimators take advantage from the network
growth although the local sample remains constant (even $n_i=1$). Indeed, the
RMSE decreases as $N$ grows.
The \EBestimator, being the optimal estimator, always outperforms the
\AHestimator.
However, as predicted by the theoretical analysis, the two estimators achieve
the same asymptotic limit as $N\rightarrow\infty$. Moreover, it is interesting
to notice that the RMSE of the two estimators practically coincide already for
$N = 16$, thus strengthening the already appealing features of the \AHestimator\
found from the theoretical analysis (i.e., easier computation and asymptotic
optimality).

\begin{figure}
\centering
\includegraphics[width=0.4\textwidth]{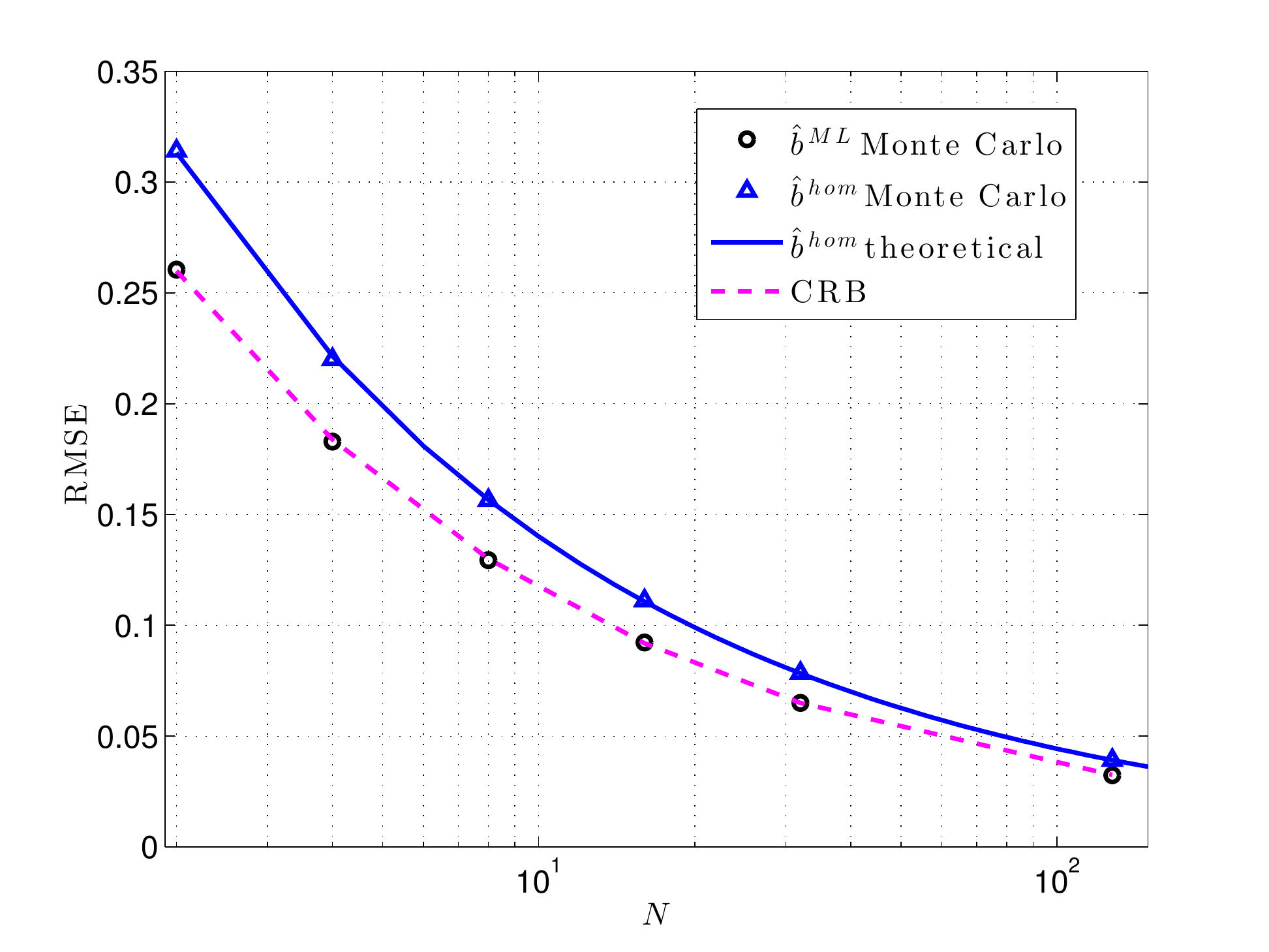}
\caption{Sample RMSE for the homogeneous (triangles markers) and ML (circles
  markers) estimators of $b$ as a function of the number of agents $N$. The
  sample values are compared with the Cramer-Rao Bound (dashed line) and the
  theoretical RMSE of the homogeneous estimator (solid line).}
\label{fig:b_asympt_N}
\end{figure}

\begin{figure}
\centering
\includegraphics[width=0.4\textwidth]{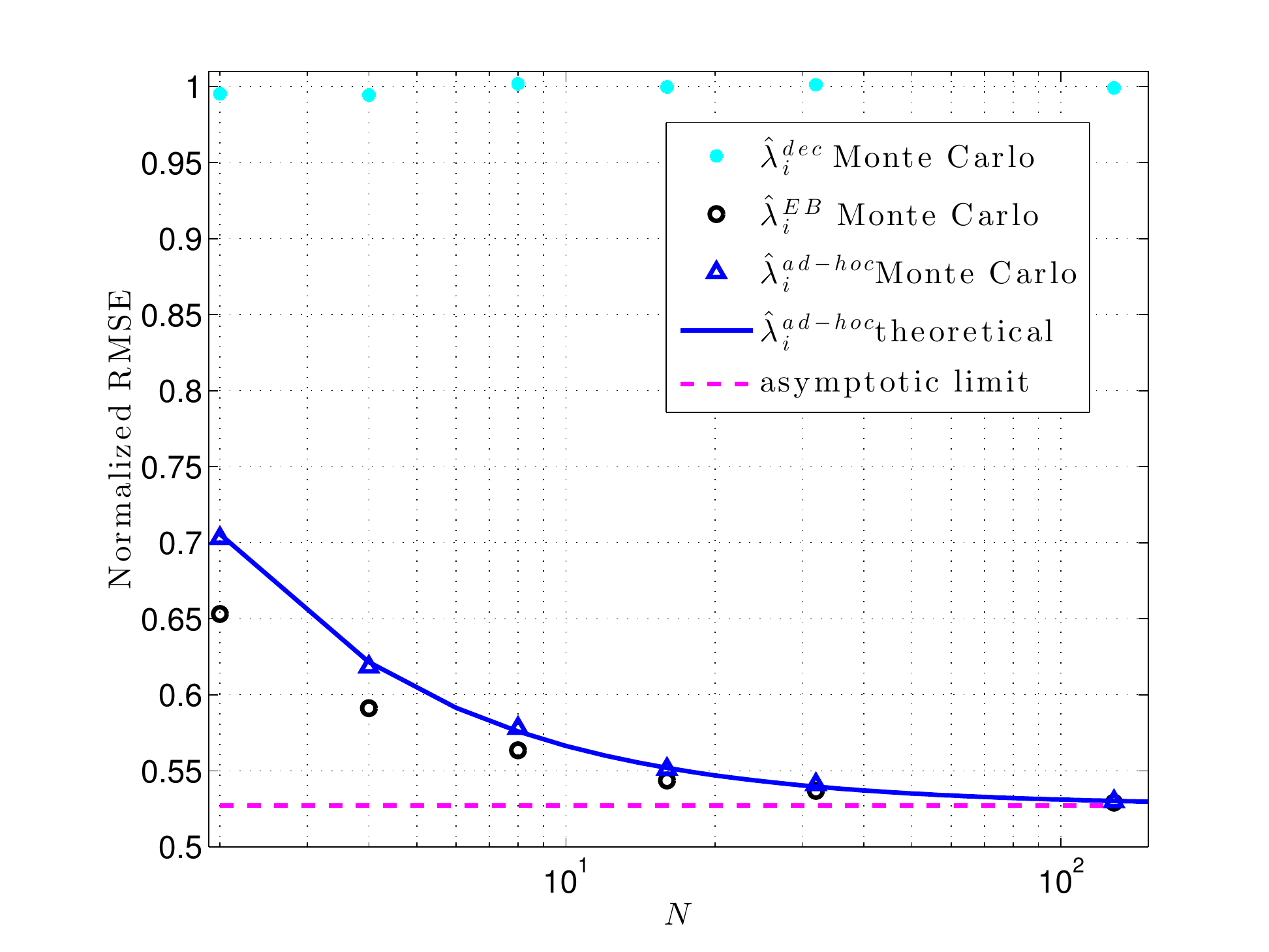}
\caption{Sample RMSE for the ad-hoc (triangle markers) and Empirical Bayes
  (circle markers) estimators normalized with respect to the decentralized
  estimator of $\lambda_i$ (dot markers) as a function of the number of agents
  $N$.  The theoretical RMSE of the ad-hoc estimator (solid line) and the
  theoretical limit (dashed line) are shown for comparison.}%
\label{fig:lambda_asympt_N}
\end{figure}

\section{Conclusions}
In this paper we have proposed a novel distributed scheme, based on a
hierarchical framework, for the Bayesian estimation of arrival rates in
asynchronous monitoring networks. The proposed distributed approach allows each
node to gain information from the network and thus outperforms the decentralized
estimator, especially when the node local information is scarce. In particular,
the distributed estimator consists of the convex combination of a global
information, computed through a distributed optimization algorithm, and a
function of the local data. Then we have proposed an \AHestimator\, that
performs closely to the optimal Empirical Bayes estimator, but is much simpler
to implement and exhibits faster (exponential) convergence. We have analyzed the
two estimators and provided expressions for mean and variance as the network
size goes to infinity, showing that in this asymptotic situation the ad-hoc
estimator achieves the same RMSE of the optimal one. Moreover, for the ad-hoc
estimator we have provided transient expressions for mean and variance. A
numerical Monte Carlo analysis has been performed to corroborate the theoretical
results and highlight the interesting features of the two distributed
estimators.

\appendix

\subsection{Proof of Lemma~\ref{lem:CRB}}
\label{sec:lem_CRB}
  The Cramer-Rao Bound is defined as $\mathrm{CRB}(b) = \frac{1}{I(b)}$, where $I(b)$ is
  the Fisher Information. The Fisher Information is obtained from the likelihood
  \eqref{eq:like} as:
\begin{equation*}
  I(b)=-\E\left[ \frac{\partial^2 \log L(\bm{y}_1, \ldots, \bm{y}_N|b)}{\partial b^2} \right], 
\end{equation*}
and, computing the derivatives, 
\begin{align*}
  I(b) &= -\E\left[ \frac{aN}{b^2} - \sum_{i=1}^N \frac{2n_ib+1}{b^2(n_ib+1)^2}
    (\sigma_i+a)\right]\\ 
  &= -\frac{aN}{b^2} + \sum_{i=1}^N \frac{2n_ib+1}{b^2(n_ib+1)^2} (\E\left[\sigma_i \right]+a).
\end{align*}
Now, the mean of $\sigma_i$ turns out to be $\E[\sigma_i ] = \E[\E[\sigma_i |
\lambda_i]]= \E[n_i \lambda_i] =n_i a b$, so that, after some manipulation
\begin{align*}
  I(b) &= \frac{a}{b^2} \sum_{i=1}^N \left[-1 + \frac{2n_ib+1}{(n_ib+1)^2} (n_i
    b+1)\right] \\
  &= \frac{a}{b^2}\sum_{i=1}^N \frac{n_ib}{n_ib+1},
\end{align*}
so that the proof follows. 

\subsection{Proof of Theorem~\ref{thm:mean_var_EB_estimator}}
The MMSE estimator of $\lambda_j$ in \eqref{eq:lambdaEB_0} can be written
  as $\lambdaEB_j= Z(\bML)Y(\sigma_j)$, where $Z(\bML)\eqdef \frac{\bML}{1+\bML n_j}$
  and $Y(\sigma_j)\eqdef a+\sigma_j$. Due to the independence between $\bML$ and
  $\sigma_j$, we have
\begin{equation}
\E[\lambdaEB_j |\lambda_j]=\E[ZY |\lambda_j]=\E[Z |\lambda_j] \E[Y
  |\lambda_j]  
\label{eq:tot_E_cond_ML}
\end{equation}
and
\begin{align}
 \! \!\!\!\VAR[\lambdaEB_j |\lambda_j]&=\E^2[Z |\lambda_j] \VAR[Y |\lambda_j]\nonumber\\
  &\quad+ \!\VAR[Z |\lambda_j] (\E^2[Y |\lambda_j] +\VAR[Y
  |\lambda_j]). 
  \label{eq:tot_VAR_cond_ML}
\end{align}
The conditional moments of $Y(\sigma_j)$ are easily obtained
\begin{equation}
  \E[Y(\sigma_j)|\lambda_j]=a+n_j \lambda_j
  \label{eq:E_Y}
\end{equation}
and
\begin{equation}
  \VAR[Y(\sigma_j) |\lambda_j]=n_j\lambda_j.
  \label{eq:VAR_Y}
\end{equation}
For the nonlinear function $Z(\bML)$ we resort to an approximate
analysis. That is, we consider the Taylor expansion for the moments of the
function $Z(\bML) \eqdef \frac{\bML}{1+\bML n_j}$ around the mean value
$\E[\bML |\lambda_j] = \E[\bML]$.
\begin{align}
  \!\!\!\E[Z(\bML) | \lambda_j] &= \E[Z(\bML)]  \nonumber \\ 
  & \approx\E \Big[ Z(\E[\bML])+ Z'(\E[\bML]) (\bML-\E[\bML])\nonumber\\ 
  & \qquad + \frac{1}{2} Z''(\E[\bML]) (\bML-\E[\bML])^2\Big], 
\end{align}
where we have neglected terms of order higher than two in the expansion of $Z$,
hence the $\approx$ symbol. Then
\begin{equation}
  \E[Z(\bML) | \lambda_j] \approx Z(\E[\bML])+ \frac{1}{2} Z''(\E[\bML]) \VAR[\bML].
\label{eq:E_Z}
\end{equation}

Observing that $Z'(s) = \frac{1}{(1+n_j s)^2}$ and $Z''(s) = -\frac{2n_j}{(1+n_j
  s)^3}$, and plugging \eqref{eq:E_Y} and \eqref{eq:E_Z} in
\eqref{eq:tot_E_cond_ML}, we get
\begin{align*}
  \E[&\lambdaEB_j| \lambda_j] \approx\\ 
  & (a+n_j\lambda_j)\left[
    \frac{\E[\bML]}{\E[\bML]n_j +1} - \frac{n_j}{(\E[\bML]n_j+1)^3}
    \VAR[\bML] \right]\!. 
\end{align*}
From the asymptotic properties of the ML estimators, we know that for
$N\rightarrow\infty$, $\E[\bML]\rightarrow b$ and $\VAR[\bML]\rightarrow
0$, so that equation \eqref{eq:E_lambdaEB_lim_N} follows.

Similarly, it turns out that
\begin{align}
  \VAR[Z(\bML)| \lambda_j] &\approx \left(Z''(\E[\bML])\right)^2 \VAR[\bML]
 \label{eq:VAR_Z}.
\end{align}
Using \eqref{eq:E_Z}-\eqref{eq:VAR_Z} into eq. \eqref{eq:tot_VAR_cond_ML} we
obtain
\begin{align*}
  &\VAR[\lambdaEB_j| \lambda_j] \approx \left( \frac{\E[\bML]}{\E[\bML]n_j +1} \right)^2 n_j\lambda_j\\
  &\quad+ \left[\frac{2n_j }{(\E[\bML]n_j +1)^3}\right]^2 \VAR[\bML] \left[(a+n_j\lambda_j)^2+n_j\lambda_j\right]\!.
\end{align*}
Using again the asymptotic properties of $\bML$, equation
\eqref{eq:VAR_lambdaEB_lim_N} follows, thus concluding the proof.

\subsection{Proof of Proposition~\ref{prop:VAR_bhom}}
\label{sec:prop_VAR_bhom}
  First, notice that $\E[\sigma_i | \lambda_i] = n_i \lambda_i$, thus 
  \begin{align*}
    \E[ \sigma ] &= \E[\E[\sigma | \bm{\lambda}]] = \E\left[ \sum_{i=1}^N \E[\sigma_i |
    \lambda_i] \right]= \E\left[\sum_{i=1}^N n_i\lambda_i\right]\\ 
&= \sum_{i=1}^N n_i a b = n a b,
\end{align*}
so that $\E[\bhom]=b$.
  
To prove the second part, recall that $\VAR[\sigma_i | \lambda_i] = n_i
\lambda_i$ and, thus, $\VAR[\sigma | \bm{\lambda} ] = \sum_{i=1}^N \VAR[\sigma_i
| \lambda_i] = \sum_{i=1}^N n_i \lambda_i$. Using the law of total variance, we have
  \[
  \begin{split}
    \VAR[ \sigma ] &= \E[\VAR[\sigma |
    \bm{\lambda}]] + \VAR[\E[\sigma | \bm{\lambda}]]\\ 
&= \E\left[\sum_{i=1}^N
    \VAR[\sigma_i | \lambda_i ] \right] + \VAR\left[ \sum_{i=1}^N \E[\sigma_i | \lambda_i ]\right]\\ 
    &= \E\left[
    \sum_{i=1}^N n_i \lambda_i \right] + \VAR\left[ \sum_{i=1}^N n_i
    \lambda_i\right]\\ 
&= a b n + a b^2
    \sum_{i=1}^N n_i^2.  
  \end{split}
  \]
  The variance of the homogeneous estimator is given by $\VAR[\bhom] =
  \VAR\left[ \frac{\sigma}{a n} \right]$, so that equation \eqref{eq:VAR_bhom}
  follows.

By Markov's inequality and Lemma~\ref{prop:VAR_bhom}
\begin{equation*}
  \mathrm{Pr}\left\{ \left| \frac{\sigma}{a n} - b \right| \geq \epsilon
  \right\} \leq \frac{1}{\epsilon} \VAR\left[ \bhom \right],
\end{equation*}
and the variance can be bounded as follows
\[
\begin{split}
  \VAR\left[ \bhom \right] &= \frac{b}{a n} +
  \frac{b^2}{a n^2} \sum_{i=1}^{N} n_i^2\\ 
&\leq \frac{b}{a N} + \frac{b^2}{a N^2}
  \sum_{i=1}^{N} \nmax^2\\
&= \frac{b}{a N} + \frac{ {\nmax}^2 b^2}{a N} \;
  \stackrel{N \rightarrow \infty}{\xrightarrow{\hspace*{1cm}} } \; 0
\end{split}
\]
so that $\bhom \rightarrow b$ in probability, thus concluding the proof.

\subsection{Proof of Lemma~\ref{lem:E_VAR_bhom_t}}
\label{sec:lem_E_VAR_bhom_t}
  To prove that $\bhom_i(t)$ is unbiased at any $t\in\naturalzero$, first let us
  recall that the aggregate states $\eta(t) = [\eta_1(t) \ldots \eta_N(t)]^T$
  and $s(t) = [s_1(t) \ldots s_N(t)]^T$ evolve according to the dynamics
  \eqref{eq:s_eta_dyn} with $s(0) = [\sigma_1 \ldots \sigma_N]^T \eqdef \sigmavec$,
  $\eta(0) = [n_1 \ldots n_N]^T \eqdef \nvec$ and $W(t)$ the column stochastic
  matrix with elements $w_{ij}(t)$.
  The evolution of $s(t)$ and $\eta(t)$ is given by \eqref{eq:s_eta_evol} where
  $\Phi(t)$ is the deterministic state transition matrix defined in
  \eqref{eq:state_trans_matrix}.

  Next, observe that the update $\eta_i(t)$ depends only on the initial value
  $n_i$ and thus is deterministic. Therefore, using the update in
  Algorithm~\ref{alg:ad-hoc_estimator}, we have
\[
\E[\bhom_i(t)] = \frac{1}{a \eta_i(t)} \E[s_i(t)]
\]
Using the evolution of the aggregate state and denoting $e_i$ the $i$th canonical
vector (e.g., $e_1 = [1 \; 0 \; \ldots \; 0]^T$), we have 
\[
\begin{split}
\E[\bhom_i(t)] &= \frac{1}{a \eta_i(t)} \E[e_i^T (\Phi(t)\sigmavec) ]\\
                     &= \frac{1}{a \eta_i(t)} e_i^T \Phi(t) \E[\sigmavec].
\end{split}
\]
Noting that $\E[\sigma_i] = a b n_i$ it follows
\[
\begin{split}
\E[\bhom_i(t)] &= \frac{1}{a \eta_i(t)} e_i^T \Phi(t) a b \, \bm{n}\\
                     &= \frac{b}{\eta_i(t)} e_i^T \eta(t) = b\\
\end{split}
\]
where the last two steps follow respectively from \eqref{eq:s_eta_evol} ($\Phi(t)
\bm{n} = \eta(t)$) and from $e_i^T \eta(t) = \eta_i(t)$.

Next, we show that the transient variance is given by equation
\eqref{eq:VAR_bhom_t}. 

Using again the update in Algorithm~\ref{alg:ad-hoc_estimator}, it holds
\[
\begin{split}
  \VAR[\bhom_i(t)] &= 
\frac{1}{a^2 \eta_i(t)^2} \VAR[e_i^T (\Phi(t)\sigmavec)].\\
\end{split}
\]
Noting that $e_i^T \Phi(t)$ is the $i$th row of $\Phi(t)$, it follows 
\[
\begin{split}
  \VAR[\bhom_i(t)] &= 
\frac{1}{a^2 \eta_i(t)^2} \VAR\left[\sum_{k=1}^N \phi_{ik}(t) \sigma_k\right],
\end{split}
\]
where $\phi_{ik}(t)$ is the element $(i,k)$ of the matrix $\Phi(t)$. By the
law of total variance we have that
\[
\begin{split}
  \VAR[\bhom_i(t)] &= 
\frac{1}{a^2 \eta_i(t)^2} \left\{ \E\left[\VAR\left[\sum_{k=1}^N \phi_{ik}(t)
      \sigma_k \;\Big |\; \lambdavec\right]\right] \right.\\
 &\qquad\left. +\VAR\left[ \E\left[\sum_{k=1}^N \phi_{ik}(t) \sigma_k  \;\Big |\;
      \lambdavec\right] \right] \right\}\\
  &= \frac{1}{a^2 \eta_i(t)^2} \left\{  \E\left[\sum_{k=1}^N \phi_{ik}(t)^2 n_k
      \lambda_k \right]\right.\\ 
    &\left.\qquad +  \VAR\left[\sum_{k=1}^N \phi_{ik}(t) n_k
      \lambda_k \right] \right\},
\end{split}
\]
where we have used $\E[\sigma_k] = n_k \lambda_k$ and $\VAR[\sigma_k] = n_k
\lambda_k$. Finally, recalling that $\E[\lambda_k] = a b$, $\VAR[\lambda_k] = a
b^2$ and the $\lambda_k$s are independent, it turns out 
\[
\begin{split}
  \VAR[\bhom_i(t)] &= \frac{b}{a \eta_i(t)^2} \left( \sum_{k=1}^N \phi_{ik}(t)^2
    n_k + b \sum_{k=1}^N \phi_{ik}(t)^2 n_k^2 \right),
\end{split}
\]
so that equation~\eqref{eq:VAR_bhom_t} follows.

To prove the asymptotic result, we work out $| \VAR[\bhom_i(t)] - \VAR[\bhom]|$
by using \eqref{eq:VAR_bhom_t} and \eqref{eq:VAR_bhom}
\[
\begin{split}
| &\VAR[\bhom_i(t)] -
\VAR[\bhom]| 
=\\ %
&=\left| \frac{b}{a} \left(\frac{\sum_{k=1}^N \phi_{ik}(t)^2
  n_k}{{\left(\sum_{k=1}^N \phi_{ik}(t) n_k\right)^2}} - \frac{1}{n}\right)\right.\\ 
&\qquad\left.+ \frac{b^2}{a} \left(\frac{
  \sum_{k=1}^N \phi_{ik}(t)^2 n_k^2}{\left(\sum_{k=1}^N \phi_{ik}(t)
    n_k\right)^2} -  \frac{\sum_{k=1}^{N} n_k^2}{n^2} \right) \right|\\
&= %
  \frac{1}{\eta_i(t)^2} \left| \frac{b}{a n} \left( n \sum_{k=1}^N \phi_{ik}(t)^2
  n_k -  \left(\sum_{k=1}^N \phi_{ik}(t) n_k\right)^2\right) \right.\\ 
&\qquad \left. + \frac{b^2}{a n^2} \left(
  n^2 \sum_{k=1}^N \phi_{ik}(t)^2 n_k^2 - \left(\sum_{k=1}^N \phi_{ik}(t)
    n_k\right)^2 \sum_{k=1}^{N} n_k^2 \right)\!\right|\\
\end{split}
\]
By using the definition of $n=\sum_{h=1}^N n_h$ and writing $\left(\sum_{k=1}^N
  \phi_{ik}(t) n_k\right)^2 = \left(\sum_{k=1}^N \phi_{ik}(t) n_k\right)\left( \sum_{h=1}^N \phi_{ih}(t) n_h\right)$
\[
\begin{split} 
|& \VAR[\bhom_i(t)] -
\VAR[\bhom]| 
=\\ %
 &= \frac{1}{\eta_i(t)^2} \left| \frac{b}{a n}
    \left( \sum_{h=1}^N n_h \sum_{k=1}^N \phi_{ik}(t)^2
  n_k \right.\right.\\ 
&\qquad\left.-  \sum_{k=1}^N \phi_{ik}(t) n_k \sum_{h=1}^N \phi_{ih}(t) n_h\right) \\ 
&\qquad + \frac{b^2}{a n^2} \left(
 \sum_{h=1}^N n_h \sum_{\ell=1}^N n_\ell \sum_{k=1}^N \phi_{ik}(t)^2 n_k^2 \right.\\ 
&\qquad\left. \left.- \sum_{h=1}^N \phi_{ih}(t)
    n_h \sum_{\ell=1}^N \phi_{i\ell}(t)
    n_\ell \sum_{k=1}^{N} n_k^2 \right)\right|\\
&= %
  \frac{1}{\eta_i(t)^2} \Bigg| \frac{b}{a n}
    \sum_{h=1}^N \sum_{k=1}^N  n_h n_k \phi_{ik}(t)
  \left( \phi_{ik}(t)
  -  \phi_{ih}(t) \right) \\ 
&\qquad + \frac{b^2}{a n^2} 
 \sum_{h=1}^N \sum_{\ell=1}^N \sum_{k=1}^N n_h n_\ell n_k^2 \left(\phi_{ik}(t)^2 - \phi_{ih}(t)
     \phi_{i\ell}(t) \right.\\ 
&\qquad + \left.\phi_{ik}(t) \phi_{i\ell}(t) - \phi_{ik}(t) \phi_{i\ell}(t) \right)\Bigg|\\
&\leq %
  \frac{1}{\eta_i(t)^2}\Bigg[ \frac{b}{a n}
    \sum_{h=1}^N \sum_{k=1}^N  n_h n_k 
 \phi_{ik}(t) \left| \phi_{ik}(t)
  -  \phi_{ih}(t) \right| \\ 
&\qquad + \frac{b^2}{a n^2} 
 \sum_{h=1}^N \sum_{\ell=1}^N \sum_{k=1}^N n_h n_\ell n_k^2
   \phi_{ik}(t) \Big( \left|\phi_{ik}(t) - \phi_{ih}(t)\right|\\ 
& \qquad + \left|\phi_{ik} - \phi_{i\ell}(t) \right| \Big) \Bigg]
\end{split}
\]
where the last inequality follows by using the triangular inequality.
Now, weak ergodicity of $W(t)$ implies that for any $i,k,h\in\until{N}$ it holds
$\left|\phi_{ik}(t) - \phi_{ih}(t)\right| \leq M \lambda^t$ for some $M>0$ and
$\lambda>0$, see, e.g., \cite[Corollary 8]{nedic2013distributed}, so that the
exponential convergence of $\VAR[\bhom_i(t)]$ follows. Then, from the definition
of the coefficient of ergodicity, it follows
\begin{align*}
| &\VAR[\bhom_i(t)] -
\VAR[\bhom]| 
\nonumber \\%
&\leq\frac{1}{\sum_{k=1}^N \phi_{ik}(t) n_k }\left( \frac{b}{a n}
   n  \delta(\Phi(t)) + \frac{b^2}{a n^2} 
  n^2 \nmax 2 \delta(\Phi(t)) \right)
\end{align*}
where we have simplified the common factor $\eta_i(t) =\sum_{k=1}^N \phi_{ik}(t)
n_k$.
Finally, by Lemma~\ref{lem:weak_ergodicity}, $\sum_{k=1}^N\phi_{ik}(t)\geq\mu$,
with $\mu>0$ so that 
\begin{align*}
| \VAR[\bhom_i(t)] -
\VAR[\bhom]| 
&\leq %
  \frac{\delta(\Phi(t)) b(1+2b\nmax)}{\mu a}
\end{align*}
thus concluding the proof.

\subsection{Proof of Theorem~\ref{thm:mean_var_AD-HOC_estimator}}

The local update of the \AHestimator\, (Algorithm~\ref{alg:ad-hoc_estimator})
can be written as $\lambdaAH_j(t)= X(\bhom_j(t)) Y(\sigma_j)$, where
$X(\bhom_j(t))\eqdef \frac{\bhom_j(t)}{1+\bhom_j(t) n_j}$ and
$Y(\sigma_j)\eqdef a+\sigma_j$. Using again the independence between
$\bhom_j(t)$ and $\sigma_j$, we can write $\E[\lambdaAH_j(t) |\lambda_j]$ and
$\VAR[\lambdaAH_j(t) |\lambda_j]$
as 
  \begin{equation}
    \E[\lambdaAH_j(t) |\lambda_j] =\E[XY |\lambda_j] =\E[X |\lambda_j] \E[Y |\lambda_j]  
    \label{eq:tot_E_cond_AH}
  \end{equation}
  and
  \begin{align}
    \!\!\!\VAR[\lambdaAH_j(t) &|\lambda_j]=\E^2[X |\lambda_j] \VAR[Y |\lambda_j]\nonumber\\
                                   &+ \VAR[X |\lambda_j] (\E^2[Y |\lambda_j] \!+\!\VAR[Y |\lambda_j]).
    \label{eq:tot_VAR_cond_AH}
  \end{align}
Considering, as in the previous theorem, the Taylor expansion for the
    moments of the function
    $X(\bhom_j(t)) \eqdef \frac{\bhom_j(t)}{1+\bhom_j(t) n_j}$ around the
    mean $\E[\bhom_j(t) |\lambda_j]=b$,
    we obtain
\begin{align}
    \E[X(&\bhom_j(t)) | \lambda_j] = \E[X(\bhom_j(t))]\nonumber\\ 
   &\approx X(\E[\bhom_j(t)])+ \frac{1}{2} X''(\E[\bhom_j(t)]) \VAR[\bhom_j(t)] \nonumber\\
  &= \frac{b}{1+n_j b}
  - \frac{n_j}{(1+n_j b)^3} \VAR[\bhom_j(t)]
  \label{eq:E_X}
\end{align}
where $\VAR[\bhom_j(t)]$ is the one given in \eqref{eq:VAR_bhom_t} and, again,
the $\approx$ symbol indicates that we have neglected higher-order terms in the
expansion of $X$.
Hence, plugging \eqref{eq:E_Y} and \eqref{eq:E_X} in \eqref{eq:tot_E_cond_AH},
equation \eqref{eq:E_cond_approx_AH} follows.
Similarly,
\begin{align}
  \VAR[X(&\bhom_j(t))| \lambda_j] \approx \left(X'(\E[\bhom_j(t)])\right)^2
  \VAR[\bhom_j(t)] \nonumber\\
  &= \left[ \frac{1}{(1+n_j b)^2} \right]^2
 \left( \frac{b}{a n} +
  \frac{b^2}{a n^2} \sum_{i=1}^{N} n_i^2 \right).
 \label{eq:VAR_X}
\end{align}
Using \eqref{eq:E_X}-\eqref{eq:VAR_X} into
\eqref{eq:tot_VAR_cond_AH}, equation \eqref{eq:VAR_cond_approx_AH} follows, thus
concluding the first part of the proof.

To prove the asymptotic results, we just need to recall from
Proposition~\ref{prop:VAR_bhom} that as $N\rightarrow\infty$
$\VAR[\bhom]\rightarrow0$, so that equations \eqref{eq:E_cond_approx_AH_asympt}
and \eqref{eq:VAR_cond_approx_AH_asympt} follow, thus concluding the proof.

\begin{IEEEbiography}
  [{\includegraphics[width=1in,height=1.25in,clip,keepaspectratio]{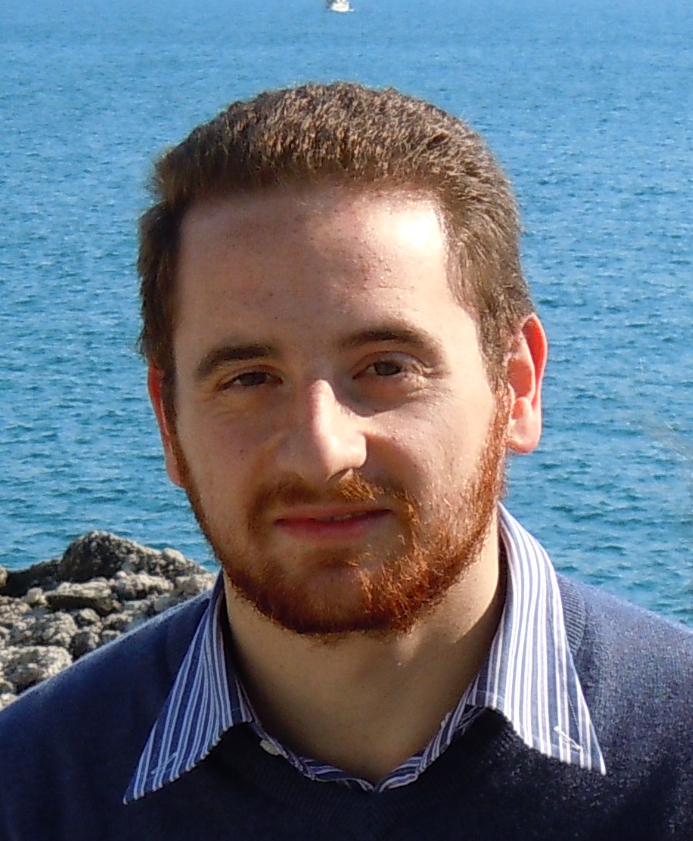}}]
  {Angelo Coluccia} (M'13) received the Eng. degree in Telecommunication
  Engineering (summa cum laude) in 2007 and the PhD degree in Information
  Engineering in 2011, both from the University of Salento, Lecce, Italy. Former
  researcher at Forschungszentrum Telekommunikation Wien, Vienna, since 2008 he
  has been engaged in research projects on traffic analysis, security and
  anomaly detection in operational cellular networks. He is currently Assistant
  Professor at the Dipartimento di Ingegneria dell'Innovazione, University of
  Salento, where he teaches the course of Telecommunication Systems.  His
  research interests are signal processing, communications and wireless
  networks, in particular cooperative sensing/estimation approaches for
  localization and other (possibly distributed) applications.
\end{IEEEbiography}

\begin{IEEEbiography}
[{\includegraphics[width=1in,height=1.25in,clip,keepaspectratio]{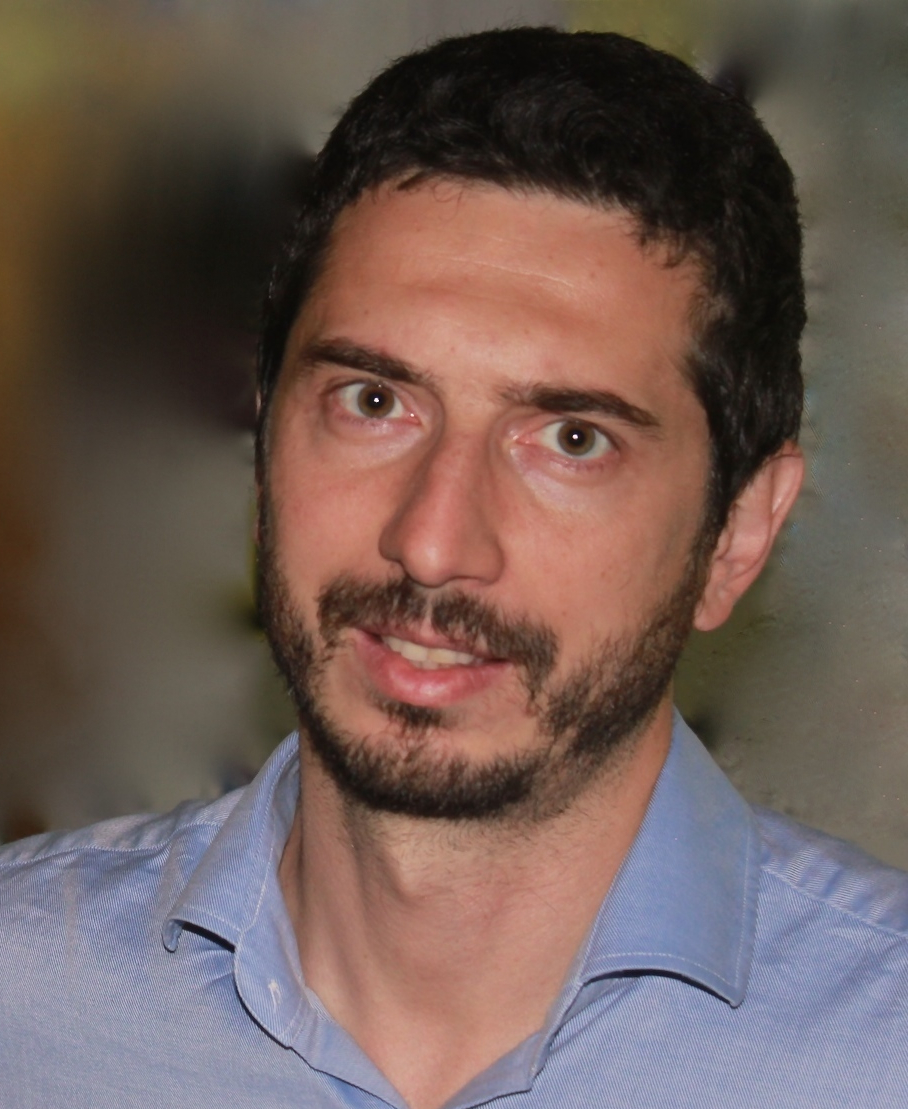}}]
{Giuseppe Notarstefano}
has been an Assistant Professor (Ricercatore) at the Universit\`a del Salento
(Lecce, Italy) since February 2007. He received the Laurea degree ``summa cum
laude'' in Electronics Engineering from the Universit\`a di Pisa in 2003 and the
Ph.D. degree in Automation and Operation Research from the Universit\`a di
Padova in April 2007. He has been visiting scholar at the Universities of
Stuttgart, California Santa Barbara and Colorado Boulder. His research interests
include distributed optimization, cooperative control in multi-agent networks,
applied nonlinear optimal control, and trajectory optimization and maneuvering
of aerial and car vehicles. He serves as an Associate Editor in the Conference
Editorial Board of the IEEE Control Systems Society and for the European Control
Conference, IFAC World Congress and IEEE Multi-Conference on Systems and
Control. He coordinated the VI-RTUS team winning the International Student
Competition Virtual Formula 2012. He is recipient of an ERC Starting Grant 2014.
\end{IEEEbiography}


\begin{thebibliography}{10}
\providecommand{\url}[1]{#1}
\csname url@samestyle\endcsname
\providecommand{\newblock}{\relax}
\providecommand{\bibinfo}[2]{#2}
\providecommand{\BIBentrySTDinterwordspacing}{\spaceskip=0pt\relax}
\providecommand{\BIBentryALTinterwordstretchfactor}{4}
\providecommand{\BIBentryALTinterwordspacing}{\spaceskip=\fontdimen2\font plus
\BIBentryALTinterwordstretchfactor\fontdimen3\font minus
  \fontdimen4\font\relax}
\providecommand{\BIBforeignlanguage}[2]{{%
\expandafter\ifx\csname l@#1\endcsname\relax
\typeout{** WARNING: IEEEtran.bst: No hyphenation pattern has been}%
\typeout{** loaded for the language `#1'. Using the pattern for}%
\typeout{** the default language instead.}%
\else
\language=\csname l@#1\endcsname
\fi
#2}}
\providecommand{\BIBdecl}{\relax}
\BIBdecl

\bibitem{coluccia2014distributed}
A.~Coluccia and G.~Notarstefano, ``Distributed bayesian estimation of arrival
  rates in asynchronous monitoring networks,'' in \emph{IEEE International
  Conference on Acoustics, Speech and Signal Processing (ICASSP)}, 2014, pp.
  5050--5054.

\bibitem{snyder1991random}
D.~L. Snyder and M.~I. Miller, \emph{Random Point Processes in Time and
  Space}.\hskip 1em plus 0.5em minus 0.4em\relax Springer, 1991.

\bibitem{heyman2003stochastic}
D.~Heyman and M.~Sobel, \emph{Stochastic Models in Operations Research:
  Stochastic optimization}, ser. Dover Books on Computer Science Series.\hskip
  1em plus 0.5em minus 0.4em\relax Dover Publications, 2003.

\bibitem{barbarossa2013distributed}
S.~Barbarossa, S.~Sardellitti, and P.~Di~Lorenzo, \emph{Distributed detection
  and estimation in wireless sensor networks}, ser. Communications and Radar
  Signal Processing.\hskip 1em plus 0.5em minus 0.4em\relax Academic Press
  Library in Signal Processing, October 2013, vol.~2, ch.~7, pp. 329--408.

\bibitem{schizas2008consensus}
I.~D. Schizas, A.~Ribeiro, and G.~B. Giannakis, ``Consensus in ad hoc {WSNs}
  with noisy linksâ {Part I}: Distributed estimation of deterministic
  signals,'' \emph{IEEE Transactions on Signal Processing}, vol.~56, no.~1, pp.
  350--364, 2008.

\bibitem{cattivelli2010diffusion}
F.~S. Cattivelli and A.~H. Sayed, ``Diffusion lms strategies for distributed
  estimation,'' \emph{IEEE Transactions on Signal Processing}, vol.~58, no.~3,
  pp. 1035--1048, 2010.

\bibitem{barbarossa2007bio}
S.~Barbarossa and G.~Scutari, ``Bio-inspired sensor network design,''
  \emph{IEEE Signal Processing Magazine}, vol.~24, no.~3, pp. 26--35, 2007.

\bibitem{garin2011survey}
F.~Garin and L.~Schenato, ``A survey on distributed estimation and control
  applications using linear consensus algorithms,'' \emph{Networked Control
  Systems}, pp. 75--107, 2011.

\bibitem{frasca2015distributed}
P.~Frasca, H.~Ishii, C.~Ravazzi, and R.~Tempo, ``Distributed randomized
  algorithms for opinion formation, centrality computation and power systems
  estimation: A tutorial overview,'' \emph{European Journal of Control}, 2015.

\bibitem{sardellitti2010fast}
S.~Sardellitti, M.~Giona, and S.~Barbarossa, ``Fast distributed average
  consensus algorithms based on advection-diffusion processes,'' \emph{IEEE
  Transactions on Signal Processing}, vol.~58, no.~2, 2010.

\bibitem{marelli2015distributed}
D.~E. Marelli and M.~Fu, ``Distributed weighted least-squares estimation with
  fast convergence for large-scale systems,'' \emph{Automatica}, vol.~51, pp.
  27--39, 2015.

\bibitem{pasqualetti2012distributed}
F.~Pasqualetti, R.~Carli, and F.~Bullo, ``Distributed estimation via iterative
  projections with application to power network monitoring,''
  \emph{Automatica}, vol.~48, no.~5, pp. 747--758, 2012.

\bibitem{nedic2013distributed}
A.~Nedic and A.~Olshevsky, ``Distributed optimization over time-varying
  directed graphs,'' \emph{arXiv preprint arXiv:1303.2289}, 2013.

\bibitem{barbarossa2007decentralized}
S.~Barbarossa and G.~Scutari, ``Decentralized maximum-likelihood estimation for
  sensor networks composed of nonlinearly coupled dynamical systems,''
  \emph{IEEE Transactions on Signal Processing}, vol.~55, no.~7, 2007.

\bibitem{chiuso2011gossip}
A.~Chiuso, F.~Fagnani, L.~Schenato, and S.~Zampieri, ``Gossip algorithms for
  simultaneous distributed estimation and classification in sensor networks,''
  \emph{IEEE Journal of Selected Topics in Signal Processing}, vol.~5, no.~4,
  pp. 691--706, 2011.

\bibitem{fagnani2011input}
F.~Fagnani, S.~M. Fosson, and C.~Ravazzi, ``Input driven consensus algorithm
  for distributed estimation and classification in sensor networks,'' in
  \emph{50th IEEE Conference on Decision and Control and European Control
  Conference (CDC-ECC)}.\hskip 1em plus 0.5em minus 0.4em\relax IEEE, 2011, pp.
  6654--6659.

\bibitem{varagnolo2010distributed}
D.~Varagnolo, G.~Pillonetto, and L.~Schenato, ``Distributed consensus-based
  bayesian estimation: sufficient conditions for performance
  characterization,'' in \emph{American Control Conference}.\hskip 1em plus
  0.5em minus 0.4em\relax IEEE, 2010, pp. 3986--3991.

\bibitem{dilorenzo2014distributed}
P.~Di~Lorenzo and S.~Barbarossa, ``Distributed least mean squares strategies
  for sparsity-aware estimation over gaussian markov random fields,'' in
  \emph{IEEE International Conference on Acoustics, Speech and Signal
  Processing (ICASSP)}, 2014, pp. 5472--5476.

\bibitem{cattivelli2008diffusion}
F.~Cattivelli, C.~G. Lopes, and A.~H. Sayed, ``Diffusion recursive
  least-squares for distributed estimation over adaptive networks,'' \emph{IEEE
  Transactions on Signal Processing}, vol.~56, no.~5, 2008.

\bibitem{mateos2009distributed}
G.~Mateos, I.~D. Schizas, and G.~B. Giannakis, ``Distributed recursive
  least-squares for consensus-based in-network adaptive estimation,''
  \emph{IEEE Transactions on Signal Processing}, vol.~57, no.~11, pp.
  4583--4588, 2009.

\bibitem{schizas2008consensus2}
I.~D. Schizas, G.~B. Giannakis, S.~I. Roumeliotis, and A.~Ribeiro, ``Consensus
  in ad hoc {WSNs} with noisy linksâ {Part II}: Distributed estimation and
  smoothing of random signals,'' \emph{IEEE Transactions on Signal Processing},
  vol.~56, no.~4, pp. 1650--1666, 2008.

\bibitem{freedman1962poisson}
D.~A. Freedman, ``Poisson processes with random arrival rate,'' \emph{The
  Annals of Mathematical Statistics}, pp. 924--929, 1962.

\bibitem{massey1996estimating}
W.~A. Massey, G.~A. Parker, and W.~Whitt, ``Estimating the parameters of a
  nonhomogeneous poisson process with linear rate,'' \emph{Telecommunication
  Systems}, vol.~5, no.~2, pp. 361--388, 1996.

\bibitem{karlis2005mixed}
D.~Karlis and E.~Xekalaki, ``Mixed poisson distributions,'' \emph{International
  Statistical Review}, vol.~73, pp. 35--58, 2005.

\bibitem{withers2011compound}
C.~Withers and S.~Nadarajah, ``On the compound poisson-gamma distribution,''
  \emph{Kybernetika}, vol.~47, no.~1, pp. 15--37, 2011.

\bibitem{casella}
E.~Lehmann and G.~Casella, \emph{Theory of Point Estimation}.\hskip 1em plus
  0.5em minus 0.4em\relax Springer, 1998.

\bibitem{zanella2012asynchronous}
F.~Zanella, D.~Varagnolo, A.~Cenedese, G.~Pillonetto, and L.~Schenato,
  ``Asynchronous newton-raphson consensus for distributed convex
  optimization,'' in \emph{3rd IFAC Workshop on Distributed Estimation and
  Control in Networked Systems (NecSys'12)}, 2012.

\bibitem{tsitsiklis1986distributed}
J.~N. Tsitsiklis, D.~P. Bertsekas, M.~Athans \emph{et~al.}, ``Distributed
  asynchronous deterministic and stochastic gradient optimization algorithms,''
  \emph{IEEE Transactions on Automatic Control}, vol.~31, no.~9, pp. 803--812,
  1986.

\bibitem{seneta2006non}
E.~Seneta, \emph{Non-negative matrices and Markov chains}.\hskip 1em plus 0.5em
  minus 0.4em\relax Springer, 2006.

\bibitem{jadbabaie2003coordination}
A.~Jadbabaie, J.~Lin, and A.~S. Morse, ``Coordination of groups of mobile
  autonomous agents using nearest neighbor rules,'' \emph{IEEE Transactions on
  Automatic Control}, vol.~48, no.~6, pp. 988--1001, 2003.

\bibitem{benezit2010weighted}
F.~B{\'e}n{\'e}zit, V.~Blondel, P.~Thiran, J.~Tsitsiklis, and M.~Vetterli,
  ``Weighted gossip: Distributed averaging using non-doubly stochastic
  matrices,'' in \emph{IEEE International Symposium on Information Theory
  Proceedings (ISIT)}.\hskip 1em plus 0.5em minus 0.4em\relax IEEE, 2010, pp.
  1753--1757.

\bibitem{vaidya2011distributed}
N.~H. Vaidya, C.~N. Hadjicostis, and A.~D. Dominguez-Garcia, ``Distributed
  algorithms for consensus and coordination in the presence of packet-dropping
  communication links-part ii: Coefficients of ergodicity analysis approach,''
  \emph{arXiv preprint arXiv:1109.6392}, 2011.

\end{thebibliography}
\end{document}